\documentclass[12pt,centertags,oneside]{amsart}
\usepackage{amssymb}
\usepackage{amsmath,amstext,amsthm,a4,amssymb,amscd}
\usepackage{typearea}
\usepackage[cal=
euler,calscaled=1.1,
bb=fourier,
scr=rsfso
]
{mathalfa}
\usepackage{epsf}
\usepackage{hyperref}
\usepackage{charter}
\usepackage{dsfont}
\usepackage{color}
\usepackage{mathabx}

\usepackage[english]{babel}
\usepackage[T1]{fontenc}

\usepackage[a4paper,left=2.5cm,right=2.5cm,top=2.5cm,
bottom=2.5cm]{geometry}

\numberwithin{equation}{section}
\allowdisplaybreaks[3]

\newcommand{\field}[1]{\mathbb{#1}}
\newcommand{\Z}{\field{Z}}
\newcommand{\R}{\field{R}}
\newcommand{\C}{\field{C}}
\newcommand{\N}{\field{N}}

 \def\cC{\mathscr{C}}

\def\cO{\mathscr{O}}

\def\mL{\mathcal{L}}

\newcommand{\boldsym}[1]{\boldsymbol{#1}}
\newcommand\bb{\boldsym{b}}

 \DeclareMathOperator{\Ker}{Ker}
 \DeclareMathOperator{\Dom}{Dom}

\DeclareMathOperator{\Bs}{Bs}

\newcommand{\bfs}{{\rm \mathbf{s}}}
\newcommand{\FS}{{\rm FS}}

\renewcommand{\P}{\field{P}}

\newcommand{\om}{\omega}

\newtheorem{thm}{Theorem}[section]
\newtheorem{lemma}[thm]{Lemma}
\newtheorem{prop}[thm]{Proposition}
\newtheorem{cor}[thm]{Corollary}
\theoremstyle{definition}
\newtheorem{rem}[thm]{Remark}
\theoremstyle{definition}

\newcommand{\ov}{\overline}

\newcommand{\var}{\varepsilon}

\newcommand{\comment}[1]{}

\begin{document}

\title[Bergman kernels and equidistribution for sequences of line bundles]
{Bergman kernels and equidistribution 
for sequences of line bundles on K\"ahler manifolds}

\author{Dan Coman}

\address{Department of Mathematics,
Syracuse University,
\newline
    \mbox{\quad}\,Syracuse, NY 13244-1150, USA}
\email{dcoman@syr.edu}
\thanks{D.\ C.\ is partially supported by the NSF Grant DMS-1700011}

\author{Wen Lu}

\address{School of Mathematics and Statistics,
Huazhong University of Science and Technology,
\newline
    \mbox{\quad}\,Wuhan 430074, China}
\email{wlu@hust.edu.cn}
\thanks{W.\ L.\ supported by National Natural Science Foundation
of China (Grant Nos. 11401232, 11871233)}

\author{Xiaonan Ma}

\address{Institut de Math\'ematiques de Jussieu-Paris Rive Gauche,
Universit\'e de Paris, CNRS, 
 F-75013
\newline
    \mbox{\quad} \,Paris, France}
\email{xiaonan.ma@imj-prg.fr}
\thanks{X.\ M.\ partially supported by
NNSFC   No.11829102 and
funded through the Institutional Strategy of
the University of Cologne within the German Excellence Initiative}

\author{George Marinescu}

\address{Universit{\"a}t zu K{\"o}ln,  Mathematisches Institut,
    Weyertal 86-90,   50931 K{\"o}ln, Germany
    \newline
    \mbox{\quad}\,Institute of Mathematics `Simion Stoilow', 
	Romanian Academy,
Bucharest, Romania}
\email{gmarines@math.uni-koeln.de}
\thanks{G.\ M.\ partially supported by 
the DFG funded projects SFB TRR 191 `Symplectic Structures in Geometry, 
Algebra and Dynamics' (Project-ID 281071066\,--\,TRR 191)
and SPP 2265 `Random Geometric Systems' (Project-ID 422743078).
}

\date{\today}

\begin{abstract}
Given a sequence of positive Hermitian holomorphic line bundles 
$(L_p,h_p)$ on a K\"ahler manifold $X$, 
we establish the asymptotic expansion of the Bergman kernel 
of the space of global holomorphic sections of $L_p$, 
under a natural convergence assumption on the sequence of curvatures 
$c_1(L_p,h_p)$. We then apply this to study the asymptotic distribution 
of common zeros of random sequences of $m$-tuples of sections of 
$L_p$ as $p\to\infty$. 
\end{abstract}

\maketitle
\setcounter{section}{-1}
\tableofcontents


\section{Introduction}\label{s0}
The familiar setting of geometric quantization
is a compact K\"ahler manifold $(X, \om)$ with K\"ahler form 
$\om$ endowed with a Hermitian holomorphic line bundle $(L, h^L)$,
called prequantum line bundle, satisfying the prequantization condition
\begin{equation}\label{0.1}
\om =\frac{\sqrt{-1}}{2\pi}R^{L}=c_{1}(L, h^{L}),
 \end{equation}
where $R^{L}$ denotes the curvature of the Chern connection on 
$(L, h^L)$ and
$c_{1}(L, h^{L})$ denotes the Chern curvature form of $(L, h^{L})$.
The existence of the prequantum line bundle $(L, h^{L})$ allows
to consider the Hilbert space of holomorphic sections $H^0(X,L)$
and construct a correspondence between smooth objects on $X$
(classical observables) and operators on $H^0(X,L)$ (quantum observables)
\cite{Ber,Do01},
stated in terms of the semi-classical limit in which
Planck's constant tends to zero.
Changing Planck's constant is equivalent to rescaling the K\"ahler form,
and this is achieved by taking tensor powers $L^{p} = L^{\otimes p}$ 
of the line bundle, since the
curvature of $L^{p}$ is $p\om$.
A pivotal role in this correspondence is played by the orthogonal projection
on $H^{0}(X, L^p)$.  Its integral kernel, called Bergman kernel, 
admits a full asymptotic
expansion as $p\to\infty$ to any order 
(cf.\ \cite{C99,DLM04,MM07,MM08,T90,Z98}).

Condition \eqref{0.1} in an integrality condition; 
a prequantum bundle exists if and only if 
the de Rham cohomology class $[\omega]$ is integral,
$[\omega]\in H^2(X,\Z)$.
What can one do in general if $\om$ is a not necessarily integral K\"ahler
form?
We can then associate to $\om$ a more general 
sequence positive line bundle $(L_p, h_p)$ 
such that their curvatures only approximate
multiples of $\omega$.
Such a sequence can
be thought as a ``prequantization'' of the nonintegral K\"ahler metric
$\omega$.

In this paper we establish the asymptotic expansion of 
the Bergman kernel of the holomorphic
space $H^{0}(X, L_{p})$ on K\"ahler manifold $X$ under 
a natural approximation assumption
of $\omega$ by the curvatures of the positive line bundles $L_p$.

Let $(X,\vartheta,J)$ be a compact K\"ahler manifold of $\dim_{\C}X=n$
with K\"ahler form $\vartheta$ and
complex structure $J$. Let $(L_p, h_p)$, $p\geqslant 1$, be 
a sequence of holomorphic Hermitian line
bundles on $X$ with smooth Hermitian metrics $h_{p}$. 
Let $\nabla^{L_p}$ be the Chern connection on
$(L_p, h_p)$ with curvature $R^{L_{p}} = (\nabla^{L_{p}})^{2}$.
Denote by $c_1(L_p, h_p)$ the Chern
curvature form of $(L_p, h_p)$.
Let $g^{TX}(\cdot, \cdot) = \vartheta(\cdot, J\cdot)$ be 
the Riemannian metric
on $TX$ induced by $\vartheta$ and $J$. The Riemannian volume 
form $dv_{X}$ has
the form $dv_X = \vartheta^{n}/n!$. We endow the space 
$\cC^{\infty}(X, L_p)$ of smooth sections of $L_p$
with the inner product 
\begin{equation}\label{e:ip}
\langle s_1,s_2\rangle_p:=\int_{X}\langle s_1,s_2\rangle_{h_p}\,
\frac{\vartheta^n}{n!}\,,\;\,s_1,s_2\in\cC^{\infty}(X, L_p),
\end{equation}
and we set $\|s\|_p^2=\langle s,s\rangle_p$.
We denote by $\mL^{2}(X,L_p)$ the completion of $\cC^{\infty}(X, L_p)$
with respect to this norm.
Let $H^{0}(X, L_p)$ be the space of holomorphic sections 
of $L_{p}$ and
let $P_{p}:\mL^{2}(X,L_p) \rightarrow H^{0}(X,L_p)$ be 
the orthogonal projection.
The integral kernel $P_p(x, x')$ $(x, x'\in X)$ of $P_p$ with respect
to $dv_{X}(x')$ is smooth and
is called the Bergman kernel. The restriction of the
Bergman kernel to the diagonal of $X$ is the Bergman kernel function
of $H^{0}(X,L_p)$, which
we still denote by $P_p$, i.e., $P_{p}(x)=P_{p}(x, x)$.

The main result of this paper is as follows.
\begin{thm}\label{t1}
Let $(X, \vartheta)$ be a compact K\"ahler manifold of $\dim_{\C}X=n$. 
Let $(L_p, h_p)$,
$p\geqslant 1$, be a sequence of holomorphic Hermitian 
line bundles on $X$ with smooth Hermitian
metrics $h_{p}$.
Let $\omega$ be a K\"ahler form on $X$ such that
\begin{equation}\label{0.6}
A^{-1}_{p}c_1(L_p, h_p) = \omega + O(A^{-a}_{p}),\ 
\text{as $p\to\infty$, in the $\cC^{\infty}$- topology,}
\end{equation}
where $a>0$, $A_p >0$ and $\lim_{p\rightarrow \infty}A_p = +\infty$.
Then as $p\to\infty$, in the $\cC^{\infty}$- topology,
\begin{equation}\label{0.7a}
\begin{split}
P_p(x)=A^{n}_{p}\bb_{0}(x)&+A^{n-1}_{p}\bb_{1}(x)
+\ldots +A^{n+ \lfloor -a \rfloor +1}_{p}
\bb_{-\lfloor -a \rfloor-1}(x)+O(A_{p}^{n-a}),
\end{split}
\end{equation}
where $\bb_{0}(x)=\omega^n/\vartheta^{n}$ and 
$\bb_{1}=\frac{1}{8\pi}(\omega^n/\vartheta^{n})r^X_\omega$,
where $r^X_\omega$ is the scalar curvature of $\omega$, 
and $\lfloor -a \rfloor$ the integer part of $-a$. 
\end{thm}
Note that the following general result was obtained
in \cite[Theorem 1.2]{CMM17}. Let $(X,\vartheta)$ be a 
compact K\"ahler manifold of dimension $n$ and $(L_p,h_p)$,
$p\geq1$, be a sequence of holomorphic Hermitian 
line bundles on $X$ with singular Hermitian
metrics $h_{p}$ that satisfy  
$c_1(L_p,h_p)\geq a_p\,\vartheta$, where $a_p>0$ and 
$\lim_{p\to\infty}a_p=\infty$\,.
If $A_p=\int_Xc_1(L_p,h_p)\wedge\vartheta^{n-1}$ denotes the mass
of the current $c_1(L_p,h_p)$, then 
$\frac{1}{A_p}\,\log P_p\to 0$ in $L^1(X,\vartheta^n)$ 
as $p\to\infty$\,. Theorem \ref{t1} refines this result
under the stronger assumptions that the metrics are smooth
and \eqref{0.6} holds. 

Our assumption \eqref{0.6} means that for any $k\in\N$, 
there exists $C_{k}>0$ such that
\begin{equation}\label{0.8}
\big|A_p^{-1}c_1(L_{p}, h_{p})-\omega\big|_{\cC^{k}}
\leqslant C_{k}A_{p}^{-a},
\end{equation}
where the $\cC^{k}$-norm is induced by the 
Levi-Civita connection $\nabla^{TX}$.
We will give several natural examples of sequences
$(L_p,h_p)$ as above. The most straightforward is
$(L_{p},h_p)=(L^{\otimes p},h^{\otimes p})$ for some fixed prequantum 
line bundle 
$(L,h)$. Then it follows from \eqref{0.1} that \eqref{0.6} holds
for $A_{p}=p$ and all $a>0$.
In this case we recover from (\ref{0.7a}) the known result
on asymptotic expansion of Bergman kernel of $H^{0}(X, L^p)$
(cf.\ \cite{C99, DLM04, MM07, T90, Z98}).
Other examples include $(L_{p},h_p)=(L^{\otimes p},h_p)$
where $h_p$ is not necessarily the product $h^p$, e.\,g.\
$h_p=h^pe^{-\varphi_p}$, with suitable weights $\varphi_p$,
or tensor powers of several bundles, see Theorem \ref{t2}.

Our approximation assumption \eqref{0.6} (or \eqref{0.8}) is natural 
in the following sense. Given a K\"ahler form $\omega$
one can first approximate its cohomology class $[\omega]\in H^2(X,\R)$
by integral classes in $H^2(X,\Z)$ by using diophantine approximation
(Kronecker's lemma) and then one constructs smooth forms
representing these approximating classes.
By \cite[Th\'eor\`eme 1.3, p.\ 57]{Laeng} condition \eqref{0.8}
holds true for any $k$, $A_{p}=p$ and $a=1+1/\beta_{2}(X)$, 
where $\beta_{2}(X)$ denotes the second Betti number of $X$,
but in general with a non-necessarily holomorphic Hermitian line bundle
$(L_p,h_p)$. In this paper we show that if there is a good diophantine
approximation with holomorphic line bundles we obtain 
corresponding good asymptotics of the Bergman kernel.

If $0<a<1$, then Theorem \ref{t1} reduces to the following result.
\begin{cor}\label{t1a}
Let $(X,\vartheta)$ be a compact K\"ahler manifold of $\dim_{\C}X=n$. 
Let $(L_p, h_p)$,
$p\geqslant 1$, be a sequence of holomorphic Hermitian line bundles
on $X$ with smooth Hermitian
metrics $h_{p}$.
Assume that there exists a K\"ahler form $\omega$ on $X$ such that
\begin{equation}\label{0.8a}
A^{-1}_{p}c_1(L_p, h_p) = \omega + O(A^{-a}_{p}),\
\text{as $p\to\infty$, in the $\cC^{\infty}$- topology,}
 \end{equation}
where $0 <a< 1$, $A_p >0$ and $\lim_{p\rightarrow \infty}
A_p = +\infty$. Then
\begin{equation}\label{0.8b}
P_p(x) = A^{n}_{p}\bb_{0}(x)+O(A_{p}^{n-a}),\
\text{as $p\to\infty$, in the $\cC^{\infty}$- topology,}
\end{equation}
where $\bb_{0}(x)= \omega^n/\vartheta^{n}$.
\end{cor}

Note that a similar result was obtained in 
\cite[Theorem 1.3]{CMM17} under different approximation
assumptions.

An interesting situation when the previous results apply is 
when $L_p$ equals a product of
tensor powers of several holomorphic line bundles, 
$L_p=F_1^{m_{1,p}}\otimes\ldots\otimes F_k^{m_{k,p}}$, 
where $\{m_{j,p}\}_p$\,, $1\leq j\leq k$, are sequences in $\N$ 
such that $m_{j,p}=r_j\,p+O(p^{1-a})$ as $p\to\infty$, where $a\geq2$
and $r_j>0$ are given. This means that $(m_{1,p},\ldots,m_{k,p})\in\N^k$ 
approximate the semiclassical ray 
$\R_{>0}\cdot(r_1,\ldots,r_k)\in\R_{>0}^k$ with a remainder 
$O(p^{1-a})$, as $p\to\infty$ (cf.\ also \cite[Corollary 5.11]{CMM17}).
\begin{thm}\label{t2}
Let $(X, \vartheta)$ be a compact K\"ahler manifold of $\dim_{\C}X=n$.
Let $(F_j , h^{F_j})$ be smooth holomorphic Hermitian line bundles on
$X$ with
$c_1(F_j , h^{F_j}) \geqslant 0$ for $1 \leqslant j\leqslant k$
and one of them is strictly positive, say, 
$c_1(F_1, h^{F_1}) \geqslant \var \vartheta$
for some $\var>0$. Let $r_j>0$, $1\leqslant j \leqslant k$,
be positive real numbers and set
$\omega =\sum_{j=1}^{k} r_{j} c_1(F_j , h^{F_j})$.
Assume that there exist sequences $\{m_{j,p}\}_{p}$, 
$1\leqslant j\leqslant k$,
in $\N$ and $a\geqslant 2$, $C>0$ such that
\begin{equation}
\Big|\frac{m_{j,p}}{p}-r_{j}\Big|\leqslant \frac{C}{p^{a}}, \quad
1 \leqslant j \leqslant k,\ for\ p > 1.
\end{equation}
Let $P_p$ be the Bergman kernel function of 
$H^{0}(X, F_{1}^{m_{1, p}}\otimes\ldots\otimes  F^{m_{k, p}}_{k})$.
Then
\begin{eqnarray}
&P_{p}(x)=p^{n}\bb_{0}(x) 
+ p^{n-1}\bb_{1}(x)+ \ldots + p^{n-k}\bb_k(x)+O(p^{n-a}),\
\\&\text{as $p\to\infty$, in the $\cC^{\infty}$- topology,}
\nonumber 
\end{eqnarray}
where $k =-\lfloor -a \rfloor -1$, $\bb_{0}(x) = \omega^n/\vartheta^{n}$ and 
$\bb_{1}=\frac{1}{8\pi}(\omega^n/\vartheta^{n})r^X_\omega$,
where $r^X_\omega$ is the scalar curvature of $\omega$.
\end{thm}
\bigskip
 
We apply Theorem \ref{t1} to study the asymptotic distribution 
of common zeros of random sequences of $m$-tuples of sections 
of $L_p$ as $p\to\infty$, see 
\cite{CMM17,CM13a,CM13b,CM15,CMN16,DMV12,DS06,SZ99,SZ08} 
for previous results and references.
Let $(X,\omega)$ be
a compact K\"ahler manifold 
of dimension $n$ and let $(L_p,h_p)$, $p\geqslant1$, 
be a sequence of Hermitian holomorphic line bundles on $X$. 
To study the equidistribution problem in a more general frame, 
we assume that  the metrics $h_p$ are of class $\cC^2$ and 
verify condition \eqref{0.8} for $k=0$. Namely, 
there exist a K\"ahler form $\omega$ on $X$ and $a>0$, $C_0>0$, 
such that for every $p\geqslant1$ we have 
\begin{equation}\label{e:T}
\big|A_p^{-1}c_1(L_{p}, h_{p})-\omega\big|_{\cC^{0}}\leqslant 
C_0A_p^{-a}\,,\:\:
\text{where $A_p >0$ and $\lim_{p\rightarrow \infty}A_p=\infty$.}
\end{equation} 
 
As before we endow the space of global holomorphic sections 
$H^0(X,L_p)$ with the inner product \eqref{e:ip}
and we set $\|s\|_p^2=\langle s,s\rangle_p$, $d_p=\dim H^0(X,L_p)$. 
Let $P_p$ be the Bergman kernel function of $H^0(X,L_p)$. 
We assume that there exist a constant $M_0>1$ and $p_0\in\N$ 
such that  
\begin{equation}\label{e:Bkf}
\frac{A_p^n}{M_0}\leqslant P_p(x)\leqslant M_0A_p^n,
\end{equation}
holds for every $x\in X$ and $p>p_0$. 
Note that, under the stronger hypothesis \eqref{0.6}, 
condition \eqref{e:Bkf} follows easily from
Theorem \ref{t1} (see \eqref{0.7a}). 

Given $m\in\{1,\ldots,n\}$ and $p\geqslant1$ 
we consider the multi-projective space 
\begin{equation}\label{e:Xpm}
{\mathbb X}_{p,m}:= \big(\P H^0(X,L_p)\big)^m,
\end{equation} 
equipped with the probability measure $\sigma_{p,m}$ 
which is the $m$-fold product of the Fubini-Study 
volume on $\P H^0(X,L_p)\simeq\P^{d_p-1}$. 
If $s\in H^0(X,L_p)$ we denote by $[s=0]$ the current of 
integration (with multiplicities) over the analytic hypersurface 
$\{s=0\}\subset X$, and we let  
\[[\bfs_p=0]:=[s_{p1}=0]\wedge\ldots\wedge[s_{pm}=0]\,,\,\text{ for 
$\bfs_p=(s_{p1},\ldots,s_{pm})\in {\mathbb X}_{p,m}$,}\]
whenever this current is well-defined (see Section \ref{S:equidist}). 
We also consider the probability space
\[({\mathbb X}_{\infty,m},\sigma_{\infty,m}):= 
\prod_{p=1}^\infty ({\mathbb X}_{p,m},\sigma_{p,m})\,.\]

In the above setting, we have the following theorem:

\begin{thm}\label{T:equidist}
Let $(X,\vartheta)$ be a compact K\"ahler manifold of dimension $n$ 
and let $(L_p,h_p)$, $p\geqslant1$, be a sequence of Hermitian 
holomorphic line bundles on $X$ with metrics $h_p$ of class $\cC^2$. 
Assume that conditions \eqref{e:T} and \eqref{e:Bkf} hold. 
Then there exist $C>0$ and $p_1\in\N$ 
such that for every $\beta>0$, $m\in\{1,\ldots,n\}$ and $p>p_1$ 
there exists a subset $E_{p,m}^\beta\subset{\mathbb X}_{p,m}\,$ 
with the following properties:

\smallskip

(i) $\sigma_{p,m}(E_{p,m}^\beta)\leqslant CA_p^{-\beta}$;

\smallskip

(ii) if $\bfs_p\in{\mathbb X}_{p,m}\setminus E_{p,m}^\beta$ then, 
for any $(n-m,n-m)$ form $\phi$ of class $\cC^2$ on $X$,
\begin{equation}\label{e:equisp}
\Big|\Big \langle\frac{1}{A_p^m}\,[\bfs_p=0]
-\omega^m,\phi\Big\rangle\Big|
\leqslant C\,\Big((\beta+1)\,\frac{\log A_p}{A_p}+
A_p^{-a}\Big)\,\|\phi\|_{\cC^2}\,.
\end{equation}
Moreover, if $\sum_{p=1}^\infty A_p^{-\beta}<+\infty$ 
then estimate \eqref{e:equisp} holds for 
$\sigma_{\infty,m}$-a.e.\ sequence 
$\{\bfs_p\}_{p\geqslant1}\in{\mathbb X}_{\infty,m}$ 
provided that $p$ is large enough.
\end{thm} 

The question of characterizing the positive closed currents on $X$ 
which can be approximated by currents of integration along analytic subsets 
of $X$, and its local version as well, are important problems 
in pluripotential theory and have many applications. 
Results in this direction are obtained in \cite{CM13b,CMN18}. 
Theorem \ref{T:equidist} shows in particular that the smooth positive 
closed form $\omega^m$ can be approximated by currents of integration 
along analytic subsets of $X$ of dimension $n-m$, for each 
$m\in\{1,\ldots,n\}$. 


The paper is organized as follows. In Section \ref{s1} 
we show that the asymptotic expansion of the Bergman kernel can be
localized. 
In Section \ref{s2} we establish the asymptotic expansion of the Bergman 
kernel near the diagonal and then prove Theorem \ref{t1}. 
The proof of Theorem \ref{T:equidist} is given in Section \ref{S:equidist}, 
using the technique of meromorphic transforms of 
Dinh and Sibony \cite{DS06}, 
as in the papers \cite{CMN16,CMN18}.

\section{Localization of the problem}\label{s1}
In this Section we show that the problem is local by using
the Lichnerowicz formula.  

\subsection{Lichnerowicz formula}

The complex structure $J$ induces a splitting 
$TX\otimes_{\R} \C =T^{(1,0)}X\oplus T^{(0,1)}X$,
where $T^{(1,0)}X$ and $T^{(0,1)}X$ are the eigenbundles 
of $J$ corresponding to the eigenvalues
$\sqrt{-1}$ and $-\sqrt{-1}$ respectively. Let $T^{\ast(1,0)}X$ 
and $T^{\ast(0,1)}X$ be the corresponding
dual bundles. Denote by $\Omega^{0,j}(X, L_p)$ the space of 
smooth $(0, j)$-forms over $X$ with values in
$L_p$ and set 
$\Omega^{0,\bullet}(X,L_p) = \oplus^{n}_{j=0}\Omega^{0,j}(X,L_p)$. 
We still denote by
$\langle\cdot, \cdot\rangle $ the fibrewise metric on
$\Lambda(T^{\ast(0, 1)}X)\otimes L_p$ induced by $g^{TX}$ 
and $h_p$.

The $L^{2}$-scalar product on $\Omega^{0,\bullet}(X,L_p)$ is given by
\eqref{e:ip}.
Let $\ov{\partial}^{L_p,\ast}$
be the formal adjoint of the Dolbeault operator $\ov{\partial}^{L_p}$
with respect to the scalar product \eqref{e:ip}.
The Dolbeault-Dirac operator is given by
\begin{equation}
D_{p}:=\sqrt{2}\Big(\ov{\partial}^{L_p}+\ov{\partial}^{L_p,\ast}\Big):
\Omega^{0, \bullet}(X, L_p) \rightarrow \Omega^{0, \bullet}(X, L_p).
\end{equation}
The Kodaira Laplacian
\begin{equation}\label{e:kdl}
\Box^{L_p}:=\ov{\partial}^{L_p}\ov{\partial}^{L_p,\ast}
+ \ov{\partial}^{L_p,\ast}\ov{\partial}^{L_p}:
\Omega^{0, \bullet}(X, L_p) \rightarrow \Omega^{0, \bullet}(X, L_p)\,,
\end{equation}
preserves the $\Z$-grading on $\Omega^{0, \bullet}(X,L_p)$.
It is an essentially self-adjoint operator on the space
$\mL^2_{0, \bullet}(X,L^p)$, 
the $L^2$-completion of $\Omega^{0, \bullet}(X,L_p)$. 
We have
\begin{equation}\label{1.3}
D^{2}_{p}=2\,\Box^{L_p}.
\end{equation}
For any $v\in TX$ with decomposition 
$v=v_{1,0}+v_{0,1}\in T^{(1,0)}X\oplus T^{(0,1)}X$,
let $v^{\ast}_{1,0}\in T^{\ast (0,1)}X$ be the metric dual
of $v_{1,0}$. Then
$c(v)=\sqrt{2}(v^{\ast}_{1,0}\wedge-i_{v_{0, 1}})$
defines the Clifford action of $v$ on
$\Lambda(T^{\ast(0, 1)}X)$, where $\wedge$ and $i$
denote the exterior and interior product respectively.

Let $\nabla^{TX}$ denote the Levi-Civita connection on
$(TX, g^{TX})$, then its induced connection on $T^{(1, 0)}X$
is the Chern connection $\nabla^{T^{(1, 0)}X}$ 
on $(T^{(1,0)}X, h^{T^{(1, 0)}X})$,
where $h^{T^{(1, 0)}X}$ is the Hermitian metric on $T^{(1,0)}X$
induced by $g^{TX}$. The
Chern connection $\nabla^{T^{(1, 0)}X}$ on $T^{(1,0)}X$ induces
naturally a connection
$\nabla^{\Lambda(T^{\ast(0, 1)}X)}$ on $\Lambda(T^{\ast(0, 1)}X)$.
Then by \cite[p. 31]{MM07} we have for an orthonormal frame 
$\{e_j\}^{2n}_{j=1}$ of $(X, g^{TX})$, 
\begin{equation}\label{1.7}
D_p =\sum^{2n}_{j=1}c(e_j)
\nabla_{e_{j}}^{\Lambda(T^{\ast(0, 1)}X)\otimes L_{p}},
\end{equation}
with
\begin{equation}\label{1.7a}
\nabla^{\Lambda(T^{\ast(0, 1)}X)\otimes L_{p}}
=\nabla^{\Lambda(T^{\ast(0, 1)}X)} \otimes\operatorname{Id}+
\operatorname{Id}\otimes\nabla^{L_{p}}.
\end{equation}
Denote by $\Delta^{\Lambda(T^{\ast(0, 1)}X)\otimes L_{p}}$
the Bochner Laplacian
on $\Lambda(T^{\ast(0, 1)}X)\otimes L_{p}$. Then
\begin{equation}
\Delta^{\Lambda(T^{\ast(0, 1)}X)\otimes L_{p}}
= -\sum^{2n}_{j=1}\Big[\nabla_{e_{j}}^{\Lambda(T^{\ast(0, 1)}X)
\otimes L_{p}}
\nabla_{e_{j}}^{\Lambda(T^{\ast(0, 1)}X)\otimes L_{p}}
-\nabla_{\nabla^{TX}_{e_{j}}e_{j}}
^{\Lambda(T^{\ast(0, 1)}X)\otimes L_{p}}\Big].
\end{equation}
Let $K_{X}=\det(T^{\ast(1, 0)}X)$ be the canonical line bundle on $X$.
The Chern connection $\nabla^{T^{(1, 0)}}X$ on $T^{(1,0)}X$ induces 
the Chern
connection $\nabla^{K^{\ast}_{X}}$ on $K^{\ast}_{X}=\det(T^{(1,0)}X)$.
Denote by $R^{K_{X}^{\ast}}$ the curvature of $\nabla^{K^{\ast}}$
and by $r^{X}$ the scalar curvature of $(X, g^{TX})$.
The Lichnerowicz formula (cf. \cite[(1.4.29)]{MM07}) reads
\begin{equation}\label{1.10}
D^{2}_{p}= \Delta^{\Lambda(T^{\ast(0, 1)}X)\otimes L_p}
+\frac{r^{X}}{4}
+\frac{1}{2}\Big(R^{L_p} +\frac{1}{2}R^{K^{\ast}_{X}}\Big)
(e_i, e_j)c(e_i)c(e_j).
\end{equation}

\subsection{Spectral gap of the Dirac operator}
As in the case of powers $L^p$ of a single line bundle
we have a spectral gap for the square $D^2_p$ of the
Dirac operator acting on $L_p$. The result and the proof
are analogous to \cite[Theorem 1.1]{MM02}, \cite[Theorem 1.5.5]{MM07}.
For a Hermitian holomorphic line bundle $(L,h)$
on $X$ set
\begin{equation}\label{1.8}
a_{L}\coloneq\inf
\left\{\frac{R_x^{L}(u, \ov{u})}{|u|_{g^{TX}}^2}:
x\in X, u\in T_{x}^{(1, 0)}X\setminus\{0\}
\right\}.
\end{equation}
Note that $a_L(x)=\inf\big\{R_x^{L}(u, \ov{u})/|u|_{g^{TX}}^2:
u\in T_{x}^{(1, 0)}X\setminus\{0\}\big\}$
is the smallest eigenvalue of the curvature form $R_x^{L}$
with respect to $g^{TX}_x$ for $x\in X$ and 
$a_L=\inf_{x\in X}a_L(x)$.

We denote by ${\rm Spec}(A)$
the spectrum of a self-adjoint operator $A$ on a Hilbert space.

\begin{thm}\label{t1.2}
Let $(X, \vartheta)$ be a compact K\"ahler manifold. There 
exists $C>0$ such that for all Hermitian holomorphic line bundles 
$(L,h)$ on $X$ the Dirac operator $D=D_L$ on $L$ satisfies the
estimate
\begin{equation}\label{1.11a}
\big\|Ds\big\|^2_{L^2} \geqslant (2 a_{L}-C)\big\|s\big\|_{L^{2}}^{2},
\:\:s\in \Omega^{>0}(X, L)
:=\bigoplus^{n}_{j=1} \Omega^{0,j}(X, L).
\end{equation}
Moreover, ${\rm Spec}(D^2) \subset \{0\} \cup [2 a_{L}-C, \infty)$.
\end{thm}
\begin{proof}
We will use the Bochner-Kodaira-Nakano formula \cite[(1.4.63)]{MM07}.
The Chern connection on $K^{\ast}_{X}\otimes L$ is given by
\begin{equation}
 \nabla^{K^{\ast}_{X}\otimes L} = \nabla^{K^{\ast}_{X}} 
 \otimes\operatorname{Id}+\operatorname{Id}\otimes\nabla^{L},
 \end{equation}
and its curvature is 
\begin{equation}
R^{K^{\ast}_{X}\otimes L} = R^{K^{\ast}_{X}}\otimes\operatorname{Id}
+\operatorname{Id}\otimes R^{L}.
\end{equation}
Let $\{w_j\}^{n}_{j=1}$ be an orthonormal frame of $T^{(1,0)}X$.
Then \cite[(1.4.63)]{MM07} reads
\begin{equation}\label{herm24b}
\Box^{L} s=\Delta^{0,\bullet}s
+R^{L\otimes K^*_X}(w_j,\overline w_k)\overline w^k
\wedge i_{\ov{w}_j}s,\quad \mbox{for} \,\, s\in\Omega^{0,\bullet}(X,L)\,.
\end{equation}
where $\Delta^{0, \bullet}$ is
a holomorphic Kodaira type Laplacian.
Since $\langle \Delta^{0, \bullet}s, s\rangle \geqslant 0$, 
\eqref{herm24b} yields
\begin{equation}\label{1.14}
\big\|\ov{\partial}^{L}s\big\|^2
+\big\|\ov{\partial}^{L, \ast}s\big\|^2
\geqslant
\big\langle R^{K^{\ast}_{X}\otimes L}(w_j ,\ov{w}_k)\ov{w}^{k}
\wedge i_{\ov{w}_j} s, s\big\rangle\,,
\quad s\in \Omega^{0,\bullet}(X,L).
\end{equation}
Since $X$ is compact there exists $C>0$ such that
\begin{equation}\label{1.14b}
\big\langle R^{K^{\ast}_{X}}(w_j, \ov{w}_k)\ov{w}^{k}
\wedge i_{\ov{w}_{j}}s, s\big\rangle
\geqslant  -C \big\|s\big\|^2\,,\quad
s\in \Omega^{0,\bullet}(X,L).
\end{equation}
By \eqref{1.8} and \eqref{1.14b} we have
\begin{equation}\label{1.16}
\big\langle R^{L}(w_j, \ov{w}_k)\ov{w}^{k}\wedge i_{\ov{w}_{j}}s, 
s\big\rangle 
\geqslant  a_{L}\big\|s\big\|^2\,,\quad s\in \Omega^{>0}(X,L).
   \end{equation}
Then \eqref{1.11a} follows immediately from \eqref{1.3} 
and \eqref{1.14}--\eqref{1.16}.

Since $X$ is compact, $D^2$ has a discrete spectrum
consisting of eigenvalues of finite multiplicity.
Let $s\in \cC^{\infty}(X, L)$ is an eigensection 
of $D^{2}$
with $D^2 s = \lambda  s$ and $\lambda \neq 0$, then
$D s \neq  0$ and
\begin{equation}
D^2(Ds) = \lambda D s.
\end{equation}
Now $Ds\in \Omega^{0,1}(X, L)$, so by \eqref{1.11a}
we have $\lambda\geqslant 2 a_{L}-C$.
\end{proof}
For a sequence $(L_p, h_p)$, $p\geqslant 1$, let us denote 
\begin{equation}\label{1.8a}
a_p\coloneq a_{L_p},
\end{equation}
with $a_{L_p}$ in (\ref{1.8}). We have thus
\begin{equation}\label{1.11}
\big\|D_ps\big\|^2_{L^2} \geqslant (2 a_{p}-C)\big\|s\big\|_{L^{2}}^{2},
\:\:s\in \Omega^{>0}(X, L_{p})
=\bigoplus^{n}_{j=1} \Omega^{0,j}(X, L_p).
\end{equation}
Note that under hypothesis \eqref{0.6} we have 
$\lim_{p\rightarrow \infty}a_{p}=\infty$.
As a consequence of Theorem \ref{t1.2} we obtain a Kodaira-Serre 
vanishing theorem for the sequence $L_p$.
\begin{cor}
Let $(X, \vartheta)$ be a compact K\"ahler manifold and 
let $(L_p, h_p)$, $p\geqslant 1$, be a sequence of holomorphic 
Hermitian line bundles on $X$ such that
$\lim_{p\to\infty}a_p=0$.
Then for $p$ large enough the Dolbeault cohomology groups of $L_p$
satisfy
\begin{equation}\label{1.5}
 H^{0,j}(X,L_{p})=0,\ \ {\rm for}\ j\neq 0.
\end{equation}
Hence the kernel of $D^{2}_{p}$
is concentrated in degree $0$ for $p$ large enough, i.e.,
\begin{equation}\label{e:kdp}
\Ker (D^{2}_{p}) = H^{0}(X,L_{p}),\ p\gg1.
\end{equation}
\end{cor}
\begin{proof}
By Hodge theory we know that
\begin{equation}\label{1.4}
\Ker D^{2}_{p}\,\big|_{\Omega^{0, j}(X,L_{p})}\simeq H^{0,j}(X,L_{p}),
\end{equation}
where $H^{0, \bullet}(X, L_{p})$ denotes the Dolbeault cohomology 
groups. Thus \eqref{1.5} follows from \eqref{1.11a}. 
Moreover, \eqref{1.4} and \eqref{1.5} yield \eqref{e:kdp}.
\end{proof}

We also need a generalization for non-compact manifolds 
of Theorem \ref{t1.2}. 
Let $(X,\vartheta)$ 
be a Hermitian manifold and let $(L,h)$ be a Hermitian 
holomorphic line bundle on $X$. 
If $(X,\vartheta)$ is complete, then the square $D^2$
of the Dirac operator on $L$ is essentially self-adjoint and we denote
by $\Dom(D^2)$ the domain of its self-adjoint extension.
\begin{thm}\label{t1.2a}
Let $(X, \vartheta)$ be a complete K\"ahler manifold such that
its Ricci curvature $c_1(K_X^*,h^{K_X^*})$ is bounded from above.
Then there exists $C>0$ such that for all Hermitian holomorphic line bundles 
$(L,h)$ on $X$ with $a_L>-\infty$ we have
\begin{equation}\label{1.11b}
\big\|Ds\big\|^2_{L^2} \geqslant (2 a_{L}-C)\big\|s\big\|_{L^{2}}^{2},
\:\:s\in\Dom(D^2)\cap
\bigoplus^{n}_{j=1}\mL^2_{0,j}(X, L).
\end{equation}
Moreover, ${\rm Spec}(D^2) \subset \{0\} \cup [2 a_{L}-C, \infty)$.
\end{thm}
\begin{proof}
The proof follows from the proof of \cite[Theorem 6.1.1]{MM07}.
\end{proof}
\subsection{Localization of the problem}

Let $a^{X}$ be the injectivity radius of $(X, g^{TX})$, and 
$\var_{0}\in (0, a^{X}/4)$.
We denote by $B^{X}(x, \var_{0})$ and $B^{T_{x}X}(0, \var_{0})$ 
the open ball in $X$ and $T_{x}X$ with the
center $x$ and radius $\var_{0}$, respectively.
Then we identify $B^{T_{x}X}(0, \var_{0})$ with $B^{X}(x, \var_{0})$
by the exponential map $Z \mapsto {\rm exp}^{X}_{x}(Z)$ 
for $Z\in T_{x}X$.
Let $f: \R \rightarrow  [0, 1]$ be a smooth even function 
such that $f(v)=1$ for $|v|\leqslant \var_{0}/2$ and
$f(v)=0$ for $|v|\geqslant \var_{0}$. Set
\begin{equation}\label{2.1}
F(a) =\Big( \int^{\infty}_{-\infty}f(v)dv\Big)^{-1}
\int^{\infty}_{-\infty} e^{iva}f(v)dv.
\end{equation}
Then $F(a)$ lies in the Schwartz space $\mathcal{S}(\R)$ and $F(0) = 1$.

\begin{prop}\label{t2.1}
For any $l,m \in \N$, $\var_{0}> 0$, there exists $C_{l,m,\var_{0}}>0$ 
such that for $p\geqslant 1$ and $x, x'\in X$,
\begin{align}\begin{split}\label{2.2}
& \big|F(D_{p})(x, x')-P_{p}(x, x')\big|_{\cC^{m}(X\times X)}
\leqslant C_{l,m,\var_{0}}A_{p}^{-l},
\\&
\big|P_{p}(x, x')\big|_{\cC^{m}(X\times X)} 
\leqslant  C_{l, m, \var_{0}}A_{p}^{-l},\ \text{if}\ d(x, x')
\geqslant \var_{0}.
\end{split}\end{align}
Here the $\cC^{m}$ norm is induced by $\nabla^{L_p}$ 
and $\nabla^{TX}$.
\end{prop}
\begin{proof} We adapt here the proof of \cite[Proposition 4.1]{DLM04},
\cite[Proposition 4.1.5]{MM07}.
For $a\in \R$, set
\begin{equation}\label{2.3}
\phi_{p}(a)=\mathds{1}_{[\sqrt{a_{p}}, \infty)}(|a|)F(a).
 \end{equation}
For $a_{p}> C$ we have by Theorem \ref{t1.2},
\begin{equation} \label{2.4}
F(D_{p})-P_{p}=\phi_{p}(D_{p}).
\end{equation}
By (\ref{2.1}), for any $m\in \N$ there exists $C_{m}>0$ such that
\begin{equation}
\sup_{a\in \R}|a|^{m}|F(a)| \leqslant C_{m}.
 \end{equation}
Since $X$ is compact there exist $\{x_i\}^{r}_{i=1}$ such that
$\{U_{i}=B^{X}(x_{i}, \var_{0})\}^r_{i=1}$ is a covering of $X$. 
We identify $B^{T_{x_{i}}X}(0, \var_{0})$ with $B^{X}(x_i, \var_{0})$ 
by the exponential map as above.
For $Z\in B^{T_{x_i}X}(0, \var_{0})$ we identify 
\[(L_p)_{Z}\cong (L_p)_{x_{i}}\,,\quad
\Lambda(T^{\ast(0, 1)}X)_{Z}\cong\Lambda(T^{\ast(0, 1)}_{x_{i}}X), \]
by parallel transport along 
the curve $[0, 1]\ni u \mapsto uZ$
with respect to the connection
$\nabla^{L_p}$ and $\nabla^{\Lambda(T^{\ast(0, 1)}X)}$, respectively.

Let $\{e_{j}\}^{2n}_{j=1}$ be an orthonormal basis
of $T_{x_{i}}X$. Let $\tilde{e}_{j}(Z)$ be the parallel transport 
of $e_{j}$
with respect to $\nabla^{TX}$ along the above curve.
Let $\Gamma^{L_{p}}$ , $\Gamma^{\Lambda(T^{\ast(0, 1)}X)}$
be the corresponding
connection forms of $\nabla^{L_{p}}$ and 
$\nabla^{\Lambda(T^{\ast(0, 1)}X)}$
with respect to any fixed frame for $L_{p}$ and 
$\Lambda(T^{\ast(0, 1)}X)$
which is parallel along the above curve under the trivialization on $U_{i}$.

Denote by $\nabla_{U}$ the ordinary differentiation operator 
on $T_{x_{i}}X$
in the direction $U$. By (\ref{1.7}),
\begin{equation}\label{2.6}
D_p =\sum^{2n}_{j=1}c(\tilde{e}_{j})\big(\nabla_{\tilde{e}_{j}}
+\Gamma^{L_{p}}(\tilde{e}_{j}) 
+ \Gamma^{\Lambda(T^{\ast(0, 1)}X)}(\tilde{e}_{j})\big).
\end{equation}
Let $\{\rho_i\}$ be a partition of unity subordinate to $\{U_{i}\}$.
For $\ell\in \N$, we define a Sobolev norm on 
the $\ell$-th Sobolev space $H^{\ell}(X,L_{p})$ by
\begin{equation}\label{2.7}
\big\|s\big\|_{H^\ell_{p}}
=\sum^{r}_{i=1}\,\sum^{\ell}_{k=0}\,\sum^{2n}_{i_{1},\ldots,i_{k}=1}
\big\|\nabla_{e_{i_{1}}}\ldots \nabla_{e_{i_{k}}}
(\rho_{i}s)\big\|_{L^{2}}.
\end{equation}
Denote by $\mathcal{R}=\sum_{j}Z_{j}e_{j}$ the radial vector field.
By \cite[(1.2.32)]{MM07}, 
$L_{\mathcal{R}}\Gamma^{L_{p}}=i_{\mathcal{R}}R^{L_{p}}$.
Set
\begin{equation}\label{2.7a}
\Gamma^{L_{p}}=\sum^{2n}_{j=1}a_{j}(Z)dZ_{j},\ \ 
a_{j}\in \cC^{\infty}(U_{i}).
\end{equation}
Then
\begin{equation}\label{2.7b}
(L_{\mathcal{R}}\Gamma^{L_{p}})_{Z}=
\sum^{2n}_{j, k=1}\Big(Z_{k}
\frac{\partial a_{j}}{\partial Z_{k}}(Z)\Big)dZ_{j}
+\sum^{2n}_{j=1}a_{j}(Z)dZ_{j}.
\end{equation}
Evaluating at the point $tZ$ yields
\begin{equation}\label{2.7c}
\frac{\partial}{\partial t}\big(ta_{j}(tZ)\big)dZ_{j}
=(L_{\mathcal{R}}\Gamma^{L_{p}})_{tZ}
=(i_{\mathcal{R}}R^{L_{p}})_{tZ}.
\end{equation}
From \eqref{2.7c} we obtain immediately $\Gamma_{0}^{L_{P}}=0$ and 
(cf. also \cite[(2-16)]{Dinh17})
\begin{equation}\label{2.8}
\Gamma^{L_{p}}_{Z}=
\int^{1}_{0}(L_{\mathcal{R}}\Gamma^{L_{p}})_{tZ}\,dt
=\int^{1}_{0}(i_{\mathcal{R}}R^{L_{p}})_{tZ}\,dt,
\end{equation}
which allows us to estimate the term $\Gamma^{L_{p}}$ in (\ref{2.6}).
From (\ref{2.6}), (\ref{2.7}) and (\ref{2.8}),
\begin{equation}\label{2.10}
\|s\|_{H^{1}_{p}}\leqslant
C\big(\|D_p s\big\|_{L^2}+A_{p}\|s\|_{L^2}\big).
\end{equation}

\noindent
Let $Q$ be a differential operator of order $m$ with scalar principal 
symbol and with compact
support in $U_{i}$, then
\begin{equation}\label{2.11}
\big[D_{p},Q\big]=\sum^{2n}_{j=1}\big[c(\tilde{e}_{j})
\Gamma^{L_{p}}(\tilde{e}_{j}), Q\big]
+\sum^{2n}_{j=1}\big[c(\tilde{e}_{j})(\nabla_{\tilde{e}_{j}}
+\Gamma^{\Lambda(T^{\ast(0, 1)}X)}(\tilde{e}_{j})),Q\big],
\end{equation}
where the sums are differential operators of orders $m-1$ and $m$,
respectively. 
By (\ref{2.10}) and (\ref{2.11}),
\begin{eqnarray}\label{2.12}
\big\|Qs\big\|_{H^{1}_{p}} &\leqslant& C\big(\|D_{p}Qs\|_{L^{2}}
+A_{p}\|Qs\|_{L^{2}}\big)
\\ &\leqslant &
C\big(\|QD_{p}s\|_{L^{2}}+A_{p}\|s\|_{H^{m}_{p}} \big).
\nonumber
\end{eqnarray}
Due to \eqref{2.12}, for every $m\in \N$ there exists $C'_m >0$ 
such that for $p \geqslant 1$,
\begin{equation}
\big\|s\big\|_{H^{m+1}_{p}} \leqslant
C'_{m}\big(\|D_{p}s\|_{H^{m}_{p}}+A_{p}\|s\|_{H^{m}_{p}}\big).
\end{equation}
This means that
\begin{equation}\label{2.14}
\|s\|_{H^{m+1}_{p}} \leqslant C'_{m} A^{m+1}_{p}
\sum^{m+1}_{j=0} A_{p}^{-j}\|D^{j}_{p}s\|_{L^2}\,.
\end{equation}
Using the Sobolev estimate \eqref{2.14} we can now
repeat the proof of \cite[Proposition 4.1]{DLM04}
and conclude the proof of Proposition \ref{t2.1}.
\end{proof}

From Proposition \ref{t2.1} and the finite propagation speed of 
solutions of hyperbolic equations
\cite[Theorem D.2.1, (D.2.17)]{MM07}, $F(D_{p})(x, x')$ only depends on
the restriction of $D_{p}$ to $B^{X}(x, \var_{0})$ and the
asymptotics $P_{p}(x, x')$ as $p\rightarrow \infty$ are localized on a
neighborhood of $x$.

\section{Asymptotic expansion of Bergman kernel}\label{s2}
In this Section we establish the asymptotic expansion of the Bergman 
kernel near
the diagonal and then prove Theorem \ref{t1}.

\subsection{Rescaling}
To get uniform estimates of the Bergman kernel in terms of $p$,
we adapt the approach of \cite[\S 2]{Dinh17} by using holomorphic 
coordinates instead of normal coordinates as in \cite[\S 4.2]{DLM04}
and \cite[\S 4.1.3]{MM07}.

Let $\psi: X\supset U\rightarrow V\subset \C^{n}$ be 
a holomorphic local chart
such that $0\in V$ and $V$ is convex (by abuse of notation, 
we sometimes identify
$U$ with $V$ and $x$ with $\psi(x)$).
Then for $x\in \frac{1}{2}V:=\{y\in \C^{n}, 2y\in V\}$, 
we will use the holomorphic coordinates
induced by $\psi$ and let $0<\var_{0}<1$ be such that 
$B(x, 4\var_{0})\subset V$ for any $x\in \frac{1}{2}V$.
We choose $\var_{0}\leqslant a^{X}/4$ in order to use the 
estimates given in the proof
of Proposition \ref{t2.1}. 

For $x\in \frac{1}{2}V$ consider the holomorphic family of
holomorphic local coordinates
\[\psi_{x}: \psi^{-1}(B(x, 4\var_{0}))\rightarrow B(0, 4\var_{0}),
\quad \psi_{x}(y):=\psi(y)-x\,.\]
The $L^{2}$-norm on $B(x, 4\var_{0})$ is given by
\begin{equation}
\|s\|^{2}_{L^{2}, x}=\int_{B(x, 4\var_{0})}|s(y)|^{2}dv_{X}(y).
\end{equation}
Let $S_{p}$ be a unitary section of $(L_{p}, h_{p})$ which is parallel
with respect to $\nabla^{L_{p}}$ along
the curve $[0, 1]\ni u\mapsto uZ$ for $|Z|\leqslant 4\var_{0}$.

\begin{lemma}\label{L:rep}
There exists a holomorphic frame $\sigma_{p}:=e^{f_{p}}S_{p}$ 
of $L_{p}$ on $B(x, 4\var_{0})$
such that
\begin{equation}\label{2.18b}
\|f_{p}\|_{\cC^{k}(B(x, 2\var_{0}))}\leqslant C_{k}\|R^{L_{p}}\|_{k+n+1},
\end{equation}
for some constant $C_{k}$ independent of $x\in \frac{1}{2}V$ and $p$.
\end{lemma}
\begin{proof}

Denote by $\Gamma^{L_{p}}$ the connection
form of $\nabla^{L_{p}}$ with respect to the frame $S_{p}$ of $L_{p}$
and by $(\Gamma^{L_{p}})^{0, 1}$ the $(0, 1)$-part of
$\Gamma^{L_{p}}$.
As $\ov{\partial}(\Gamma^{L_{p}})^{0, 1}=0$,
by \cite[Chapter VIII, Theorem 6.1 and (6.4)]{Demailly},
there exists
 $f_{p}\in \cC^{\infty}(B(x, 4\var_{0}))$ satisfying
\begin{equation}\label{2.17}
\ov{\partial}f_{p}=-(\Gamma^{L_{p}})^{0, 1},
\end{equation}
and
\begin{equation}\label{2.18}
\|f_{p}\|_{L^{2}, x}\leqslant c_{1}\|\Gamma^{L_{p}}\|_{L^{2}, x},
\end{equation}
where $c_{1}$ is a constant independent of $x\in \frac{1}{2}V$ and $p$.
Using elliptic estimate, we have
\begin{equation}\label{2.18a}
\|f_{p}\|_{k+1, x}\leqslant c_{2, k}\big(\|\ov{\partial}f_{p} \|_{k, x}
+\|f_{p}\|_{L^{2}, x}\big),
\end{equation}
where $\|\cdot\|_{k, x}$ denotes the Sobolev norm on the Sobolev space 
${\bf H}^{k}(B(x, 4\var_{0}))$
and $c_{2, k}$ is a constant independent of $x\in \frac{1}{2}V$ and $p$.
Denote by $\varphi_{p}$ be the real part of $f_{p}$. From (\ref{2.17})
we know that $\sigma_{p}:=e^{f_{p}}S_{p}$ forms a holomorphic
frame of $L_{p}$ on $B(x, 4\var_{0})$ with 
\begin{equation}\label{ma2.5a}
|\sigma_{p}|_{h_{p}}^{2}(Z)=e^{2\varphi_{p}(Z)}.
\end{equation}
The estimate (\ref{2.18b}) follows from 
(\ref{2.8}), (\ref{2.17}), (\ref{2.18}), (\ref{2.18a})
and Sobolev embedding theorem.
\end{proof}

\begin{rem}
Note that on a Stein manifold $M$ we have 
$H^1(M,\cO^*)\cong H^2(M,\Z)$
due to Cartan's theorem B (see e.\,g.\ \cite[p.\,201]{Ho1}), thus
any holomorphic line bundle $L$ over a Stein contractible
manifold, for example a coordinate ball, is holomorphically trivial
(this is due to Oka \cite{Oka}). Lemma \ref{L:rep} gives
a a proof with estimates of this result over a coordinate ball.

\end{rem}

Consider the holomorphic family of holomorphic trivializations
of $L_{p}$ associated with the coordinate $\psi_{x}$ and the frame 
$\sigma_{p}$.
These trivializations are given by
\[\Psi_{p, x}: L_{p}|_{\psi^{-1}(B(x, 4\var_{0}))}
\rightarrow B(0, 4\var_{0})\times \C\]
with $\Psi_{p, x}(y, v_{p}):=(\psi_{x}(y), v_{p}/\sigma_{p}(y))$
for $v_{p}$
a vector in the fiber of $L_{p}$ over the point $y$.

Consider a point $x_{0}\in \frac{1}{2}V$.
Denote by $\varphi_{p, x_{0}}=\varphi_{p}\circ \psi^{-1}_{x_{0}}$
the function $\varphi_{p}$ in (\ref{ma2.5a})
in local coordinate $\psi_{x_{0}}$.
Denote by $\varphi^{[1]}_{p, x_{0}}$ and $\varphi^{[2]}_{p, x_{0}}$
the first
and second order Taylor expansion of $\varphi_{p, x_{0}}$, i.e.,
\begin{eqnarray}
 \varphi^{[1]}_{p, x_{0}}(Z)&:=&
\sum^{n}_{j=1}\Big(\frac{\partial \varphi_{p}}{\partial z_{j}}(x_{0})z_{j}
+\frac{\partial \varphi_{p}}{\partial \ov{z}_{j}}(x_{0})\ov{z}_{j}\Big),
\\
 \varphi^{[2]}_{p, x_{0}}(Z)&:=&
{\rm Re}\sum^{n}_{j, k=1}
\Big(\frac{\partial^{2}\varphi_{p}}{\partial z_{j}\partial z_{k}}(x_{0})
z_{j}z_{k}
+\frac{\partial^{2}\varphi_{p}}{\partial z_{j}\partial \ov{z}_{k}}(x_{0})
z_{j}\ov{z}_{k}\Big),
\nonumber
\end{eqnarray}
where we write $z=(z_{1}, \ldots, z_{n})$ the complex coordinate of $Z$.

Let $\rho: \R \rightarrow [0, 1]$ be a smooth even function such that
\begin{equation}
\rho(t) =1\ \ {\rm if}\ \ |t| < 2;\ \rho(t)=0\ \ {\rm if}\ \ |t| > 4.
\end{equation}
We denote in the sequel $X_{0}=\R^{2n}\simeq T_{x_{0}}X$ 
and equip $X_{0}$ with the metric
$g^{TX_{0}}(Z):= g^{TX}(\rho(|Z|/\var_{0})Z)$. 
Now let $0<\var<\var_{0}$ to be determined and define
\begin{equation}\label{2.20}
\phi_{\var, p}(Z):=\rho(|Z|/\var)\varphi_{p, x_{0}}(Z)
+\big(1-\rho(|Z|/\var)\big)\big(\varphi_{p}(x_{0})
+\varphi^{[1]}_{p, x_{0}}(Z)+\varphi^{[2]}_{p, x_{0}}(Z)\big).
\end{equation}
Let $h_{\var}^{L_{p, 0}}$ be the metric on $L_{p, 0}=X_{0}\times \C$ 
defined by
\begin{equation}
|1|^{2}_{h^{L_{p, 0}}_{\var}}(Z):=e^{-2\phi_{\var, p}(Z)}.
\end{equation}
Let $\nabla_{\var}^{L_{p, 0}}$ be the Chern connection on
$(L_{p, 0}, h^{L_{p, 0}}_{\var})$ and
$R_{\var}^{L_{p, 0}}$ be the curvature of $\nabla_{\var}^{L_{p, 0}}$.
By (\ref{0.8}) and (\ref{2.18b}), there exists $C>0$ independent of 
$p$ such that
for $|Z|\leqslant 4\var$, $0\leqslant j\leqslant 2$, we have
\begin{equation}\label{2.21a}
\Big|f_{p, x_{0}}(Z)-\big(f_{p}(x_{0})+f^{[1]}_{p, x_{0}}(Z)
+f^{[2]}_{p, x_{0}}(Z)\big)\Big|_{\cC^{j}}
\leqslant C A_{p}|Z|^{3-j}.
\end{equation}
From (\ref{0.8}) and (\ref{1.8a}), we may assume that
$a_{p}/A_{p}\geqslant \mu_{0}$ holds for all $p\in \N^{\ast}$,
here $\mu_{0}$ is a constant depending only on $\omega$.
By (\ref{2.20})--(\ref{2.21a}), there exists $0<\var<\var_{0}$ 
small enough
such that the following estimate holds for every $x_{0}\in U$:
\begin{equation}\label{2.21b}
\inf\Big\{\sqrt{-1}R_{\var, Z}^{L_{p, 0}}(u, Ju)\big/|u|^{2}_{g^{TX_{0}}}:
u\in T_{Z}X_{0}\setminus\{0\}\ {\rm and}\ Z\in X_{0} \Big\}
\geqslant \frac{4}{5}a_{p}.
\end{equation}
In the sequel we fix $\var>0$ small such that (\ref{2.21b}) holds.
Let
\begin{equation}\label{2.13a}
D^{X_{0}}_{p}=\sqrt{2}\Big(\ov{\partial}^{L_{p, 0}}
+(\ov{\partial}^{L_{p, 0}})^{\ast}\Big)
\end{equation}
be the Dirac-Dolbeault operator on $X_{0}$ associated to the above data,
where $(\ov{\partial}^{L_{p, 0}})^{\ast}$ is the adjoint of 
$\ov{\partial}^{L_{p, 0}}$
with respect to the metrics $g^{TX_{0}}$ and $h_{\var}^{L_{p, 0}}$.
Note that over the ball $B(x_{0}, 2\var)$, $D_{p}$ is just
the restriction of $D^{X_{0}}_{p}$.

Let $\nabla^{T^{(1, 0)}X_{0}}$ be the holomorphic Hermitian connection
on $(T^{(1, 0)}X_{0}, h^{T^{(1, 0)}X_{0}})$ with curvature 
$R^{T^{(1, 0)}X_{0}}$.
It induces naturally a connection $\nabla^{T^{(0, 1)}X_{0}}$ 
on $T^{(0, 1)}X_{0}$.
Set $\widetilde{\nabla}^{TX_{0}}
=\nabla^{T^{(1, 0)}X_{0}}\oplus\nabla^{T^{(0, 1)}X_{0}}$.
Then $\widetilde{\nabla}^{TX_{0}}$ is a connection on 
$TX_{0}\otimes_{\R}\C$.

Let $T_{0}$ be the torsion of the connection 
$\widetilde{\nabla}^{TX_{0}}$ and
$T_{0, as}$ be the anti-symmetrization of the tensor
$V, W, Y\rightarrow \langle T_{0}(V, W), Y\rangle$.
Let $\nabla^{Cl_{0}}$ be the Clifford connection on
$\Lambda(T^{\ast(0, 1)}X_{0})$ (cf. \cite[(1.3.5)]{MM07}).
Define the operator $^{c}(\cdot)$ on
$\Lambda(T^{\ast}X_{0}\otimes_{\R}\C)$  by
$^{c}(e^{i_{1}}\wedge\ldots\wedge e^{i_{j}})
=c(e_{i_{1}})\ldots c(e_{i_{j}})$
for $1\leqslant i_{1}<\ldots<i_{j}\leqslant 2n$.
Set
\begin{equation}
\nabla_{U}^{A_{0}}=\nabla_{U}^{Cl_{0}}
-\frac{1}{4}\ ^{c}(i_{U}T_{0, as}).
\end{equation}
Then as explained in \cite[(1.4.27)-(1.4.28)]{MM07}, 
$\nabla^{A_{0}}$ preserves the $\Z$-grading on 
$\Lambda(T^{\ast(0, 1)}X_{0})$. 
Let $\nabla^{A_{0}\otimes L_{p, 0}}$ be the connection
on $\Lambda(T^{\ast(0, 1)}X_{0})\otimes L_{p, 0}$
induced by $\nabla^{A_{0}}$ and $\nabla^{L_{p, 0}}$ as in (\ref{1.7a}).
Denote by $\Delta^{A_{0}\otimes L_{p, 0}}$ the Bochner
Laplacian on $\Lambda(T^{\ast(0, 1)}X_{0})\otimes L_{p, 0}$
associated to $\nabla^{A_{0}\otimes L_{p, 0}}$.
By \cite[(1.2.51), (1.4.29)]{MM07}, we have
\begin{equation}\label{2.14a}
\begin{split}
(D^{X_{0}}_{p})^{2}=\Delta^{A_{0}\otimes L_{p, 0}}
&+\frac{r^{X_{0}}}{4}
+ \ ^{c}\Big(R^{L_{p, 0}}+\frac{1}{2}{\rm Tr}[R^{T^{(1, 0)}X_{0}}]\Big)
\\ 
&-\frac{1}{4}\ ^{c}(dT_{0, as})-\frac{1}{8}\big|T_{0, as}\big|^{2},
\end{split}
\end{equation}
where the norm $|A|$ for $A\in \Lambda^{3}(T^{\ast}X_{0})$ is given by
$|A|^{2}=\sum_{i<j<k}|A(e_{i}, e_{j}, e_{k})|^{2}$.
By Theorem \ref {t1.2a} 
we get from (\ref{2.21b}) the existence of $C>0$ such that
for any $p\in \N^{\ast}$,
\begin{equation}\label{2.21}
{\rm Spec}\big((D^{X_{0}}_{p})^{2}\big) \subset \{0\}
\cup [a_{p}-C, \infty).
\end{equation}
Note that from (\ref{2.13a}), $(D^{X_{0}}_{p})^{2}$ preserves 
the $\Z$-grading on $\Omega^{0, \bullet}(X_{0}, L_p)$. 

Let $S_{p, x_{0}}$ be the unitary section of $(L_{p, 0}, h_{p, 0})$
which is
parallel with respect to $\nabla_{\var}^{L_{p, 0}}$ along the curve
$[0, 1]\ni u\rightarrow uZ$ for any $Z\in X_{0}$.
The unitary frame $S_{p, x_{0}}$ provides an isometry 
$L_{p, 0}\simeq \C$. Let $P^{0}_{p}$ be the orthogonal projection from
$\cC^{\infty}(X_{0}, L_{p, 0}) \simeq \cC^{\infty}(X_{0}, \C)$
on $\Ker D^{X_{0}}_{p}$, and let $P^{0}_{p}(x, x')$
be the smooth kernel of
$P^0_{p}$ with respect to the volume form $dv_{X_{0}}(x')$.

\begin{prop}
For any $l,m\in \N$, there exists $C_{l,m}>0$ such that for 
$x, x' \in B^{T_{x_{0}}X}(0, \var)$,
\begin{equation}\label{2.22}
\Big|P^{0}_{p}(x,x)-P_{p}(x, x)\Big|_{\cC^{m}} 
\leqslant C_{l,m} A_{p}^{-l}.
\end{equation}
\end{prop}
\begin{proof} Using (\ref{2.1}) and (\ref{2.21}), we know that 
	$P^{0}_{p}-F(D_p)$ verifies
also (\ref{2.2}) for $x, x' \in B^{T_{x_{0}}X}(0, \var)$,
thus we get (\ref{2.22}).
\end{proof}

Now under the natural identification ${\rm End}(L_{p}) \simeq \C$
(which does not depend on $S_{p, x_{0}}$), we will consider 
$(D^{X_0}_{p})^{2}$ acting
on $\cC^{\infty}(X_{0}, \C)$.
Let $dv_{TX}$ be the Riemannian volume form of 
$(T_{x_{0}}X, g^{T_{x_{0}}X})$.
Let $\kappa(Z)$ be the smooth positive function defined by the equation
\begin{equation}
dv_{X_{0}}(Z)=\kappa(Z)dv_{TX}(Z),
\end{equation}
with $\kappa(0)=1$.
For $s\in \cC^{\infty}(\R^{2n}, \C)$,
$Z\in \R^{2n}$ and $t = \frac{1}{\sqrt{A_{p}}}$, set
\begin{align}\begin{split}\label{2.23}
(\delta_{t}s)(Z) =& s(Z/t),
\\
\nabla_{t,\bullet} =& \delta^{-1}_{t} t\kappa^{1/2}
\nabla^{L_{p, 0}}\kappa^{-1/2}\delta_{t},
 \\
\mathscr{L}^t_2 =& \delta^{-1}_{t}t^{2}
\kappa^{1/2}(D^{X_{0}}_{p})^{2} \kappa^{-1/2}\delta_{t}.
\end{split}
\end{align}

\subsection{Asymptotics of the scaled operators}
Let $\{w_{j}\}^{n}_{j=1}$ be an orthonormal basis of 
$T_{x_{0}}^{(1, 0)}X$.
Then
\begin{equation}\label{2.24}
e_{2j-1}=\frac{1}{\sqrt{2}}(w_{j}+\ov{w}_{j})\ \ {\rm and}\ \
e_{2j}=\frac{\sqrt{-1}}{\sqrt{2}}(w_{j}-\ov{w}_{j}), j=1, \ldots, n,
\end{equation}
form an orthonormal basis of $T_{x_{0}}X$.
Set
\begin{equation}
\begin{split}\label{2.25}
\nabla_{0, \bullet}
 =& \nabla_{\bullet}+\frac{1}{2}\gamma_{x_{0}} (Z, \cdot),\ 
 \text{with}\ \gamma=-2\pi\sqrt{-1}\omega,
\\
\mathscr{L}^{0}_{2} =&
-\sum^{2n}_{j=1}(\nabla_{0, e_{j}})^{2}
-\gamma_{x_{0}}(w_{j}, \ov{w}_{j}).
\end{split} \end{equation}

\begin{lemma} \label{t2.3}
The following holds as $t\rightarrow 0$\,:
\begin{equation}\label{2.26}
\nabla_{t, \bullet}=\nabla_{0, \bullet}
+\mathscr{O}(t^{{\rm min}(1, 2a)}), \ \
\mathscr{L}^{t}_{2}=\mathscr{L}^{0}_{2}
+\mathscr{O}(t^{{\rm min}(1, 2a)}).
\end{equation}
\end{lemma}
\begin{proof}
By (\ref{2.14a}), when we restrict on $\cC^{\infty}(\R^{2n}, \C)$, we get
\begin{equation}\begin{split}
(D^{X_{0}}_{p})^{2} =& \Delta^{A_{0}\otimes L_{p, 0}}
-R^{L_{p, 0}}(\tilde{w}_{j}, \ov{\tilde{w}}_{j})
\\ & +\frac{r^{X_{0}}}{4}
-\frac{1}{2}\text{Tr}\big[R^{T^{(1, 0)}X_{0}}\big]
(\tilde{w}_{j}, \ov{\tilde{w}}_{j})
-\frac{1}{4}\ ^{c}(dT_{0, as})-\frac{1}{8}\big|T_{0, as}\big|^{2},
\end{split}\end{equation}
where $\tilde{w}_{j}$ denotes the parallel transport of $w_{j}$
along the curve $[0, 1]\ni u\rightarrow uZ$.
Let $\Gamma^{A_{0}}$ be the connection form of $\nabla^{A_{0}}$ 
with
respect to the orthonormal frame of $\Lambda(T^{\ast(0, 1)}X_{0})$
which is parallel along the above curve.
Denote by $\Gamma^{L_{p, 0}}$ the connection form 
of $\nabla^{L_{p, 0}}$
with respect to the frame $S_{p, x_{0}}$.
Set $g_{ij}(Z) = g^{TX_{0}}(e_i, e_j)(Z) = \langle e_i, e_j\rangle|_{Z}$
and let $(g^{ij}(Z))$ be the inverse of the matrix $(g_{ij}(Z))$ 
and $\nabla_{e_{i}}^{TX_{0}}e_{j}=\Gamma^{k}_{ij}(Z)e_{k}$ for the 
Levi-Civita connection $\nabla^{TX_{0}}$ on $(TX_{0}, g^{TX_{0}})$.
Then
\begin{align}\begin{split}\label{2.29}
\nabla_{t,\bullet} =&\, \kappa^{\frac{1}{2}}(tZ)
\big(\nabla_{\bullet}+t\Gamma_{tZ}^{A_{0}}
+t \Gamma^{L_{p, 0}}_{tZ}\big)\kappa^{-\frac{1}{2}}(tZ),
\\
\mathscr{L}^{t}_{2}=&\, g^{ij}(tZ)
\big(\nabla_{t, e_{i}}\nabla_{t,e_j}
-t\Gamma^{k}_{ij}(tZ)\nabla_{t,e_{k}}\big)
- t^{2} R^{L_{p, 0}}_{tZ}(\tilde{w}_{j}, \ov{\tilde{w}}_{j})
\\ &
+\frac{r_{tZ}^{X_{0}}}{4}
-\frac{1}{2}\text{Tr}\big[R_{tZ}^{T^{(1, 0)}X_{0}}\big]
(\tilde{w}_{j}, \ov{\tilde{w}}_{j})
-\frac{t^{2}}{4}\ ^{c}\big((dT_{0, as})_{tZ}\big)
-\frac{t^{2}}{8}\big|(T_{0, as})_{tZ}\big|^{2}.
\end{split} \end{align}
If $|tZ| < 2\var$, then $\rho(|tZ|/\var) = 1$ and $D^{X_{0}}_{p}=D_{p}$,
in particular, $\Gamma^{A_{0}}_{tZ}=0$ on 
$\Lambda^{0}(T^{\ast(0, 1)}X_{0})=\C$, 
$T_{0}=0, T_{0, as}=0$ and $\frac{r^{X_{0}}}{4}
-\frac{1}{2}\text{Tr}\big[R^{T^{(1, 0)}X_{0}}\big]
(\tilde{w}_{j}, \ov{\tilde{w}}_{j})=0$.
Thus for $|tZ| < 2\var$, the operators 
$\nabla_{t,\bullet}, \mathscr{L}^{t}_{2}$ are given
by
\begin{align}\begin{split}\label{2.30}
\nabla_{t,\bullet} =&\, \kappa^{\frac{1}{2}}(tZ)
\big(\nabla_{\bullet}+t \Gamma^{L_{p}}_{tZ}\big)
\kappa^{-\frac{1}{2}}(tZ),\\
\mathscr{L}^{t}_{2}=&\, g^{ij}(tZ)
\big(\nabla_{t, e_{i}}\nabla_{t,e_j}
-t\Gamma^{k}_{ij}(tZ)\nabla_{t,e_{k}}\big)
- t^{2} R^{L_{p}}_{tZ}(\tilde{w}_{j}, \ov{\tilde{w}}_{j}).
\end{split}\end{align}
By \cite[(1.2.31)]{MM07},
\begin{equation}
\Gamma^{L_{p}}_{Z}(e_j)
=\frac{1}{2}R_{x_{0}}^{L_{p}}(Z, e_j)
+\mathscr{O}\big(|Z|^{2}|\Gamma^{L_{p}}|_{\cC^{2}}\big),
\end{equation}
which implies
\begin{equation}\label{2.32}
t\Gamma^{L_{p}}_{tZ}(e_j)
=\frac{t^{2}}{2}R_{x_{0}}^{L_{p}}(Z, e_j)
+\mathscr{O}\big(t^{3}|\Gamma^{L_p}|_{\cC^{2}}\big).
\end{equation}
By   (\ref{2.25}), \eqref{0.8} turns to
\begin{equation}\label{0.9}
\Big|\frac{R^{L_{p}}}{A_p}-\gamma\Big|_{\cC^{k}}
\leqslant \frac{2\pi C_{k}}{A^{a}_{p}},
\end{equation}
By (\ref{0.9}),
\begin{equation}\label{2.33}
\Big|t^2R^{L_p}-\gamma\Big|_{\cC^{k}}\leqslant 2\pi C_{k}t^{2a}.
\end{equation}
Combining \eqref{2.8}, \eqref{2.32} and \eqref{2.33} yields
\begin{equation}\label{2.35}
t\Gamma^{L_{p}}_{tZ}(e_j)=\frac{1}{2}\gamma_{x_{0}}(Z, e_{j})
+\mathscr{O}(t^{{\rm min}(1, 2a)}).
\end{equation}
By (\ref{2.33}),
\begin{equation}\label{2.36}
t^{2} R^{L_{p}}_{tZ}(\tilde{w}_{j}, \ov{\tilde{w}}_{j})=
\gamma_{x_{0}}(w_{j}, \ov{w}_{j})+\mathscr{O}(t^{{\rm min}(1, 2a)}).
\end{equation}

\noindent
Moreover,
\begin{equation}\label{2.37}
 \kappa(Z) =\det(g_{ij}(Z))^{1/2}.
 \end{equation}
 By \cite[(1.2.19)]{MM07},
 \begin{equation}\label{2.38}
 g_{ij}(Z)=\delta_{ij}+\mathscr{O}(|Z|^{2}).
 \end{equation}
Then (\ref{2.26}) follows from (\ref{2.30}) and (\ref{2.35})--(\ref{2.38}).
The proof of Lemma \ref{t2.3} is completed.
\end{proof}

To prove Theorem \ref{t1}, we need a refinement of Lemma \ref{t2.3}.
Note that $a=a_{0}+a_{1}$ with $a_{0}=-\lfloor -a \rfloor -1 \in \N$ 
and $0<a_{1}\leqslant 1$.
\begin{thm} \label{t3.1} 
The following holds as $t\rightarrow 0$: 
If $\frac{1}{2}<a_{1}\leqslant 1$\,, then
\begin{equation}\label{3.1a}
\mathscr{L}^{t}_{2}=\mathscr{L}^{0}_{2}
+\sum^{2a_{0}+1}_{r=1}t^{r}\mathcal{O}_{r}+\mathscr{O}(t^{2a}).
\end{equation}
If $0<a_{1}\leqslant \frac{1}{2}$\,, then
\begin{equation}\label{3.1b}
\mathscr{L}^{t}_{2}=\mathscr{L}^{0}_{2}
+\sum^{2a_{0}}_{r=1}t^{r}\mathcal{O}_{r}+\mathscr{O}(t^{2a}).
\end{equation}
\end{thm}

\begin{proof}
Consider the Taylor expansion
\begin{equation}\label{3.3a}
\Gamma^{L_{p}}_{Z}(e_j)=\sum^{k}_{r=1}\sum_{|\alpha|=r}
\big(\partial^{\alpha}\Gamma^{L_{p}}\big)_{x_{0}}(e_{j})
\frac{Z^{\alpha}}{\alpha!}
+\mathscr{O}\big(|Z|^{k+1}|\Gamma^{L_{p}}|_{\cC^{k+1}}\big).
\end{equation}
From \cite[(1.2.30)]{MM07} and (\ref{3.3a}), we obtain
\begin{equation}\label{3.3}
\Gamma^{L_{p}}_{Z}(e_j)
=\sum^{k}_{r=1}\frac{1}{r+1}\sum_{|\alpha|=r-1}
\big(\partial^{\alpha}R^{L_{p}}\big)_{x_{0}}(Z, e_{j})
\frac{Z^{\alpha}}{\alpha!}
+\mathscr{O}\big(|Z|^{k+1}|\Gamma^{L_{p}}|_{\cC^{k+1}}\big).
\end{equation}
Combining (\ref{2.8}), (\ref{0.9}) and (\ref{3.3}) yields
\begin{equation}\label{3.5}
t\Gamma^{L_{p}}_{tZ}(e_j)=\sum^{k}_{r=1}
\frac{t^{r-1}}{r+1}\sum_{|\alpha|=r-1}
\big(\partial^{\alpha}\gamma\big)_{x_{0}}(Z, e_{j})
\frac{Z^{\alpha}}{\alpha!}
+\mathscr{O}(t^{{\rm min}(2a, k)}).
\end{equation}
If $\frac{1}{2}<a_{1}\leqslant 1$, then we take $k=2a_{0}+2$ in \eqref{3.5}
and we obtain
\begin{equation}\label{3.6a}
t\Gamma^{L_{p}}_{tZ}(e_j)=\sum^{2a_{0}+2}_{r=1}
\frac{t^{r-1}}{r+1}\sum_{|\alpha|=r-1}
\big(\partial^{\alpha}\gamma\big)_{x_{0}}(Z, e_{j})
\frac{Z^{\alpha}}{\alpha!}
+\mathscr{O}(t^{2a}).
\end{equation}
If $0<a_{1}\leqslant \frac{1}{2}$, then we take
$k=2a_{0}+1$ in \eqref{3.5} and obtain 
\begin{equation}\label{3.6c}
t\Gamma^{L_{p}}_{tZ}(e_j)=\sum^{2a_{0}+1}_{r=1}
\frac{t^{r-1}}{r+1}\sum_{|\alpha|=r-1}
\big(\partial^{\alpha}\gamma\big)_{x_{0}}(Z, e_{j})
\frac{Z^{\alpha}}{\alpha!}
+\mathscr{O}(t^{2a}).
\end{equation}
Then \eqref{3.1a} and \eqref{3.1b}) follow
from (\ref{2.30}), (\ref{2.36}), (\ref{3.6a}) and (\ref{3.6c}).
\end{proof}
%

\subsection{Bergman kernel}
Now we discuss the eigenvalues and eigenfunctions of 
$\mathscr{L}^{0}_{2}$ in detail.
We choose $\{w_{j}\}^{n}_{j=1}$ an orthonormal basis of
$T^{(1,0)}_{x_{0}}X$ such that
\begin{equation}\label{2.39}
\gamma_{x_{0}}(w_j, \ov{w}_j) = a_j,\ \  a_j>0.
\end{equation}
Let $\{w^j\}^{n}_{j=1}$ be its dual basis.
Then $\{e_{j}\}^{2n}_{j=1}$ given by (\ref{2.24}) forms 
an orthonomal basis of $T_{x_{0}}X$.
We use the coordinates
on $\R^{2n}\simeq T_{x_{0}}X$ induced by $e_{j}$ as
\begin{equation}
\R^{2n} \ni(Z_{1}, \ldots, Z_{2n}) 
\longmapsto \sum^{2n}_{j=1}Z_{j}e_{j} \in T_{x_{0}}X.
 \end{equation}
In what follows we also introduce the complex coordinates
$z = (z_1, \ldots, z_n)$ on $\C^{n}\simeq \R^{2n}$.
Thus $Z = z + \ov{z}$,
and $w_{j} =\sqrt{2}\frac{\partial}{\partial z_{j}}$,
$\ov{w}_{j}=\sqrt{2}\frac{\partial}{\partial \ov{z}_{j}}$.
We will also identify $z$ to $\sum_{j} z_{j}\frac{\partial}{\partial z_{j}}$
and $\ov{z}$ to $\sum_{j}\ov{z}_{j}\frac{\partial}{\partial \ov{z}_{j}}$
when we consider $z$ and $\ov{z}$ as vector fields. Remark that
\begin{equation}
\Big|\frac{\partial}{\partial z_{j}}\Big|^{2}=
\Big|\frac{\partial}{\partial \ov{z}_{j}}\Big|^{2}=\frac{1}{2},\ \
{\rm so\ that}\ \ |z|^2 = |\ov{z}|^2 =\frac{1}{2}|Z|^2.
 \end{equation}
It is very useful to rewrite $\mathscr{L}^{0}_{2}$ by
using the creation and annihilation operators. Set
\begin{equation}
b_{j}=-2\nabla_{0, \frac{\partial}{\partial z_{j}}}, \ \
b^{+}_{j} = 2\nabla_{0, \frac{\partial}{\partial \ov{z}_{j}}},\ \
b=(b_1, \ldots, b_{n}).
\end{equation}
Then by (\ref{2.25}) and (\ref{2.39}), we have
\begin{equation}
b_j =-2\frac{\partial}{\partial z_{j}}+\frac{1}{2}a_{j}\ov{z}_{j}, \ \
b^{+}_{j}=2\frac{\partial}{\partial \ov{z}_{j}}+\frac{1}{2}a_jz_{j}.
\end{equation}
Then
\begin{equation}
\mathscr{L}^{0}_{2}=\sum^{n}_{j=1}b_{j}b^{+}_{j}.
\end{equation}
Let $\mathscr{P}: \big(\mL^2(\R^{2n}), \| \cdot \|_{\mL^{2}}\big)
\rightarrow\Ker(\mathscr{L}^{0}_{2})$ be the orthogonal projection. 
Denote by $\mathscr{P}(x, y)$
the Schwartz kernel of $\mathscr{P}$. By \cite[Theorem 4.1.20]{MM07},
\begin{equation}
\mathscr{P}(Z, Z') =\prod^{n}_{j=1} \frac{a_{j}}{2\pi}
{\rm exp}\Big[-\frac{1}{4}\sum_{j}a_j\big(|z_j|^{2}+|z'_j|^{2}
-2z_j\ov{z}'_{j}\big)\Big]
\end{equation}
In particular,
\begin{equation}\label{2.45}
\mathscr{P}(0, 0) =\prod^{n}_{j=1} \frac{a_{j}}{2\pi}
=\frac{\omega^{n}}{\vartheta^{n}}\,\cdot
\end{equation}

\subsection{Proof of Theorem \ref{t1}}

Denote by $\langle\cdot, \cdot\rangle$ and $\|\cdot\|_{0}$ 
the inner product and the $L^{2}$-norm
on $\cC^{\infty}(X_{0}, \C)$ induced by $g^{TX_{0}}$. 
For $s\in \cC_{0}^{\infty}(X_{0}, \C)$, set
\begin{equation}\begin{split}
\|s\|^{2}_{t, 0}&:=\|s\|^{2}_{0}=\int_{\R^{2n}}|s(Z)|^{2}dv_{TX}(Z),\\
\|s\|^{2}_{t, m}&:=\sum^{m}_{l=0}\,\sum^{2n}_{j_{1}, \ldots, j_{l}=1}
\|\nabla_{t, e_{j_{1}}}\ldots \nabla_{t, e_{j_{l}}}s\|^{2}_{t, 0}\,.
\end{split}
\end{equation}
By (\ref{2.33}), we have the following analogue of 
\cite[Theorem 4.1.9]{MM07}.
\begin{thm}\label{t3.2}
There exist $C_1, C_2, C_3>0$ such that for $t\in (0, 1)$ 
and any
$s, s'\in \cC_{0}^{\infty}(\R^{2n}, \C)$,
\begin{equation}\begin{split}\label{3.7}
&\big\langle \mathscr{L}_{2}^t s, s\big\rangle_{t}
 \geqslant
C_{1}\big\|s\big\|^{2}_{t,1}-C_{2}\big\|s\big\|^{2}_{t,0}\,,
\\
&\Big|\big\langle \mathscr{L}^{t}_{2}s, s'\big\rangle_{t, 0} \Big|
 \leqslant
C_3\big\|s\big\|_{t,1}\big\|s'\big\|_{t,1}\,.
\end{split} \end{equation}
\end{thm}

\begin{proof} The analogue of \cite[(4.1.39)]{MM07} holds:
\begin{equation}\begin{split}\label{3.8}
\big\langle \mathscr{L}^{t}_{2}s, s\big\rangle_{t, 0}
=
\big\|\nabla_{t}s\big\|^{2}_{t, 0}
-t^{2}\Big\langle \delta^{-1}_{t}
\big( & R^{L_{p, 0}}(\tilde{w}_{j}, \ov{\tilde{w}}_{j})
-\frac{r^{X_{0}}}{4}
+\frac{1}{2}\text{Tr}\big[R^{T^{(1, 0)}X_{0}}\big]
(\tilde{w}_{j}, \ov{\tilde{w}}_{j})
\\ \ \ \  & +\frac{1}{4}\ ^{c}(dT_{0, as})
+\frac{1}{8}|T_{0, as}|^{2}\big)s, s\Big\rangle.
\end{split} \end{equation}
By (\ref{2.33}) and (\ref{3.8}), we obtain the first inequality 
of (\ref{3.7}).
From (\ref{2.29}), we get the second inequality of (\ref{3.7}).
\end{proof}

\begin{proof}[Proof of Theorem \ref{t1}]
From Theorem \ref{t3.2}, we can obtain the analogue of 
\cite[Theorems 4.1.10\,--\,4.1.12]{MM07}
exactly the same way as \cite[Theorems 4.1.10\,--\,4.1.12]{MM07}
follow from \cite[Theorems 4.1.9]{MM07}.
Recall that $a_{p}/A_{p}\geqslant \mu_{0}>0$ for all $p\in \N^{\ast}$ large enough.
By (\ref{2.21}) and (\ref{2.23}), there exists $t_{0}\in (0, 1)$
such that for $t\in (0, t_{0})$,
\begin{equation}
{\rm Spec}(\mathscr{L}^{t}_{2})
\subset \{0\}\cup\Big[\,\frac{1}{2}\,\mu_{0}, +\infty\Big).
\end{equation}
Let $\delta$ be the counterclockwise oriented circle in 
$\C$ of center $0$ and radius $\mu_{0}/4$.
Then $(\lambda-\mathscr{L}^{t}_{2})^{-1}$ exists for 
$\lambda\in \delta$.

For $m\in \N$, let $\mathcal{Q}^{m}$ be the set of
operators 
$\{\nabla_{t, e_{i_{1}}}\ldots\nabla_{t, e_{i_{j}}}\}_{j\leqslant m}$.
For $k, r\in \N$, set
\begin{equation}
I_{k, r}:=\Big\{({\bf k}, {\bf r})=\big\{(k_{i}, r_{i})\big\}^{j}_{i=0},\ 
\sum^{j}_{i=0}k_{i}=k+j,\
\sum^{j}_{i=1}r_{i}=r,\ k_{i}, r_{i}\in \N^{\ast} \Big\}.
\end{equation}
Then there exist $a^{\bf k}_{\bf r}\in \R$ such that
\begin{eqnarray}
A^{\bf k}_{\bf r}(\lambda, t)&=&
(\lambda-\mathscr{L}^{t}_{2})^{-k_{0}}\frac{\partial^{r_{1}} 
\mathscr{L}^{t}_{2}}{\partial t^{r_{1}}}
(\lambda-\mathscr{L}^{t}_{2})^{-k_{1}}
\ldots
\frac{\partial^{r_{j}} \mathscr{L}^{t}_{2}}{\partial t^{r_{j}}}
(\lambda-\mathscr{L}^{t}_{2})^{-k_{j}},
 \\
\frac{\partial^{r}}{\partial t^{r}}(\lambda-\mathscr{L}^{t}_{2})^{-k}
&=&\sum_{({\bf k}, {\bf r})\in I_{k, r}} 
a^{\bf k}_{\bf r}A^{\bf k}_{\bf r}(\lambda, t).
\nonumber \end{eqnarray}
The analogue of \cite[Theorems 4.1.13\,--\,4.1.14]{MM07} is as follows. 
For any $m\in \N$,
$k>2(m+r+1), ({\bf k}, {\bf r})\in I_{k, r}$, 
there exist $C>0$, $N\in \N$ such that for any
$\lambda\in \delta$,
$t\in (0, t_{0}]$, $Q, Q'\in \mathcal{Q}^{m}$,
\begin{equation}
\big\|Q A^{\bf k}_{\bf r}(\lambda, t)Q's\big\|_{t, 0}\leqslant
C \big(1+|\lambda|\big)^{N}
\sum_{|\beta|\leqslant 2r}\|Z^{\beta}s\|_{t, 0}.
\end{equation}
Moreover, if $\frac{1}{2}<a_{1}\leqslant 1$, then for
$r\in \{0, 1, \ldots, 2a_{0}+1\}$
and any $k>0$, there exist $C>0$, $N\in \N$ such that for 
$t\in [0, t_{0}]$,
$\lambda\in \delta$, we have
\begin{eqnarray}\label{3.8a}
& &\Big\|\frac{\partial^{r} \mathscr{L}^{t}_{2}}{\partial t^{r}}
-\frac{\partial^{r} \mathscr{L}^{t}_{2}}{\partial t^{r}}\Big|_{t=0}
\Big\|_{t, -1}
\leqslant C t\sum_{|\alpha|\leqslant r+3}\big\|Z^{\alpha}s\big\|_{0, 1}.
\nonumber \\
& & \Big\|\frac{\partial^{r} }{\partial t^{r}}
(\lambda-\mathscr{L}^{t}_{2})^{-k}
-\sum_{({\bf k}, {\bf r})\in I_{k, r}}a^{\bf k}_{\bf r}
A_{\bf r}^{\bf k}(\lambda, 0)\Big\|_{0, 1}
 \\ & &  \hspace{2cm}
\leqslant C t(1+|\lambda|^{2})^{N}
\sum_{|\alpha|\leqslant 4r+3}\|Z^{\alpha}s\|_{0, 0}.
\nonumber
\end{eqnarray}

\noindent
If $0<a_{1}\leqslant \frac{1}{2}$, then for 
$r\in \{0, 1, \ldots, 2a_{0}-1\}$ we have similar estimates
to \eqref{3.8a} and for $r=2a_{0}$ and $k>0$,
there exist $C>0$, $N\in\N$, such that for $t\in [0, t_{0}]$,
$\lambda\in \delta$, we have
\begin{eqnarray}
& & \Big\|\frac{\partial^{r} \mathscr{L}^{t}_{2}}{\partial t^{r}}
-\frac{\partial^{r} \mathscr{L}^{t}_{2}}{\partial t^{r}}
\Big|_{t=0}\Big\|_{t, -1}
\leqslant C t^{2a_{1}}\sum_{|\alpha|\leqslant r+3}
\big\|Z^{\alpha}s\big\|_{0, 1},
\nonumber \\
& & \Big\|\frac{\partial^{r} }{\partial t^{r}}
(\lambda-\mathscr{L}^{t}_{2})^{-k}
-\sum_{({\bf k}, {\bf r})\in I_{k, r}}a^{\bf k}_{\bf r}
A_{\bf r}^{\bf k}(\lambda, 0)\Big\|_{0, 1}
 \\ & & \hspace{2cm}
\leqslant C t^{2a_{1}} (1+|\lambda|^{2})^{N}
\sum_{|\alpha|\leqslant 4r+3}\|Z^{\alpha}s\|_{0, 0}.
\nonumber \end{eqnarray}

By \cite[Theorems 4.1.16\,--\,4.1.18 and 4.1.21]{MM07}, we obtain the
following analogue of \cite[Theorem 4.1.24]{MM07}: for any 
$m\in \N$, $q>0$,
there exists $C>0$ such that for $Z, Z'\in T_{x_{0}}X$, 
$|Z|, |Z'|\leqslant q/\sqrt{A_{p}}$,
\begin{eqnarray}\label{3.13}
& & \Big|\frac{1}{A^{n}_{p}}P_{p}(Z, Z')
-\sum^{2a_{0}+1}_{r=0}\mathscr{F}_{r}(\sqrt{A_{p}}Z, \sqrt{A_{p}}Z')
\kappa^{-\frac{1}{2}}(Z)
\kappa^{-\frac{1}{2}}(Z')A_{p}^{-\frac{r}{2}}\Big|_{C^{m}(X)}
 \\ & & \hspace{4cm}
\leqslant C A^{-a_{0}-1}_{p},\ \ {\rm for}\ \ \frac{1}{2}< a_{1}
\leqslant 1;
\nonumber \\
& & \Big|\frac{1}{A^{n}_{p}}P_{p}(Z, Z')
-\sum^{2a_{0}}_{r=0}\mathscr{F}_{r}(\sqrt{A_{p}}Z, \sqrt{A_{p}}Z')
\kappa^{-\frac{1}{2}}(Z)
\kappa^{-\frac{1}{2}}(Z')A_{p}^{-\frac{r}{2}}\Big|_{C^{m}(X)}
\nonumber \\ & & \hspace{4cm}
\leqslant C A^{-a_{0}-\frac{1}{2}}_{p},\ \ {\rm for}\ \  0<a_{1}
\leqslant \frac{1}{2};
\nonumber
\end{eqnarray}
where $\mathscr{F}_{r}(Z, Z')=J_{r}(Z, Z')\mathscr{P}(Z, Z')$ and
$J_{r}(Z, Z')$ are polynomials in $Z$, $Z'$ with the same parity as $r$
and $J_{0}(Z, Z')=1$. Set now $Z=Z'=0$ in (\ref{3.13}). 
Then (\ref{0.7a}) follows
from (\ref{2.45}) and (\ref{3.13}). The proof of Theorem \ref{t1}
is completed.
\end{proof}


\section{Equidistribution of zeros of random sections}\label{S:equidist}

In this section we prove Theorem \ref{T:equidist}. 
Assume throughout this section the setting of Theorem \ref{T:equidist} 
and let $m\in\{1,\ldots,n\}$.  We will denote by $\omega_{_\FS}$ 
the Fubini-Study form on a projective space $\P^d$, 
normalized so that $\omega_{_{\FS}}^d$ is a probability measure.

Let us start by introducing notation and recalling some facts 
needed for the proof. If $\{S^p_j\}_{j=1}^{d_p}$ is an 
orthonormal basis of $H^0(X,L_p)$ then the Bergman kernel function 
$P_p$ of $H^0(X,L_p)$ is given by 
\begin{equation}\label{e:Bk}
P_p(x)=\sum_{j=1}^{d_p}|S^p_j(x)|_{h_p}^2\,,\,\;x\in X.
\end{equation}

Let $U$ be a contractible Stein open set in $X$ and write 
$S^p_j=f^p_je_p$, where $e_p$ is a local holomorphic frame 
of $L_p$ and $f^p_j$ is a holomorphic function on $U$. 
The \emph{Fubini-Study current} $\gamma_p$ of $H^0(X,L_p)$
is defined by 
\begin{equation}\label{e:FSdef}
\gamma_p\,|_U=\frac{1}{2}\,dd^c\log\sum_{j=1}^{d_p}|f^p_j|^2\,,  
\end{equation}
where $d=\partial+\overline\partial$, $d^c=
\frac{1}{2\pi i}(\partial -\overline\partial)$. 
These are positive closed currents of bidegree $(1,1)$, 
smooth away from the base locus $\Bs H^0(X,L_p)$ of $H^0(X,L_p)$. 
We have 
\begin{equation}\label{e:FSB}
\gamma_p=c_1(L_p,h_p)+\frac{1}{2}\,dd^c\log P_p\,.
\end{equation}
Let $\Phi_p:X\dashrightarrow\P^{d_p-1}$ be the Kodaira map defined 
by the basis $\{S^p_j\}_{j=1}^{d_p}$, so
\begin{equation}\label{e:Kodaira}
\Phi_p(x)=[f^p_1(x):\ldots:f^p_{d_p}(x)]\,\text{ for }x\in U.
\end{equation}
Then $\gamma_p=\Phi_p^*(\omega_{_{\FS}})$.

If $s\in H^0(X,L_p)$ we denote by $[s=0]$ the current of integration
(with multiplicities) along the analytic hypersurface $\{s=0\}$. 
One has the Lelong-Poincar\'e formula (see \cite[Theorem 2.3.3]{MM07})
\begin{equation}\label{e:LP}
[s=0]=c_1(L_p,h_p)+dd^c\log|s|_{h_p}\,.
\end{equation}
Recall that ${\mathbb X}_{p,m}= \big(\P H^0(X,L_p)\big)^m$, 
$d_p=\dim H^0(X,L_p)$. Set 
\begin{equation}\label{e:d_pm}
d_{p,m}:=\dim{\mathbb X}_{p,m}=m(d_p-1).
\end{equation}
Let $\pi_k:{\mathbb X}_{p,m}\to\P H^0(X,L_p)$ 
be the canonical projection onto the $k$-th factor. 
We endow ${\mathbb X}_{p,m}$ 
with the K\"ahler form 
\[\omega_{p,m}:= c_{p,m}\big(\pi_1^*\omega_{_{\FS}}+
\ldots+\pi_m^*\omega_{_{\FS}}\big),\]
where the constant $c_{p,m}$ is chosen so that 
$\omega_{p,m}^{d_{p,m}}=\sigma_{p,m}$ is a probability measure on 
${\mathbb X}_{p,m}$. It follows that  
\begin{equation}\label{e:c_p}
c_{p,m}= \left(\frac{\big((d_p-1)!\big)^m}{d_{p,m}!}\right)^{1/d_{p,m}}.
\end{equation}

\begin{lemma}\label{L:Bertini} 
In the hypotheses of Theorem \ref{T:equidist}, the following hold
for $p>p_0$: 

(i) $\gamma_p$ are smooth $(1,1)$ forms on $X$.

(ii) For $\sigma_{p,m}$-a.e.\ 
$\bfs_p=(s_{p1},\ldots,s_{pm})\in {\mathbb X}_{p,m}$ 
we have that the analytic set 
$\{s_{pi_1}=0\}\cap\ldots\cap\{s_{pi_k}=0\}$ 
has pure dimension $n-k$ for each $1\leqslant k\leqslant m$ and 
$1\leqslant i_1<\ldots<i_k\leqslant m$. In particular the current 
$[\bfs_p=0]:=[s_{p1}=0]\wedge\ldots\wedge[s_{pm}=0]$ 
is well defined and is equal to the current 
of integration with multiplicities over the common zero set 
$\{\bfs_p=0\}:=\{s_{p1}=0\}\cap\ldots\cap\{s_{pm}=0\}$.
\end{lemma}

\begin{proof} By \eqref{e:Bkf} we have $P_p(x)>0$ for all 
$x\in X$ and $p>p_0$, hence $\Bs H^0(X,L_p)=\emptyset$ 
and $(i)$ follows from \eqref{e:FSB}. Since 
$\Bs H^0(X,L_p)=\emptyset$ for $p>p_0$, \cite[Proposition 4.1]{CM15} 
implies that, for $\sigma_{p,m}$-a.e.\ 
$\bfs_p=(s_{p1},\ldots,s_{pm})\in{\mathbb X}_{p,m}$, 
the analytic hypersurfaces $\{s_{p1}=0\},\ldots,\{s_{pm}=0\}$ 
are in general position, i.e.\ 
$\{s_{pi_1}=0\}\cap\ldots\cap\{s_{pi_k}=0\}$ 
has dimension at most $n-k$ for each $1\leqslant k\leqslant m$ and 
$1\leqslant i_1<\ldots<i_k\leqslant m$. Hence 
\begin{align}\label{ma3.7}
R:=[s_{pi_1}=0]\wedge\ldots\wedge[s_{pi_k}=0]
\end{align}
is a well defined positive closed current of bidegree $(k,k)$ 
by \cite[Corollary 2.11]{Dem93}, supported in the set 
$\{s_{pi_1}=0\}\cap\ldots\cap\{s_{pi_k}=0\}$. 
Moreover, by the Lelong-Poincar\'e formula \eqref{e:LP},
\[\int_XR\wedge\vartheta^{n-k}
=\int_Xc_1(L_p,h_p)^k\wedge\vartheta^{n-k}>0\,.\]
So $\{s_{pi_1}=0\}\cap\ldots\cap\{s_{pi_k}=0\}\neq\emptyset$, 
hence it has pure dimension $n-k$. The last assertion of $(ii)$ 
now follows from \cite[Corollary 2.11, Proposition 2.12]{Dem93}.
\end{proof}


The proof of Theorem \ref{T:equidist} uses results of Dinh 
and Sibony \cite[Section 3.1]{DS06} on meromorphic transforms. 
As in \cite[Example 3.6 (c)]{DS06}, \cite[Section 4.2]{CMN16}, \cite{DMM} 
we consider the meromorphic transform
$\Phi_{p,1}$
from $X$ to ${\mathbb P}H^0(X,L_p)$ defined by its graph 
$\Gamma_{p,1}=\big\{(x,s)\in X\times{\mathbb P}H^0(X,L_p):\,s(x)=0\big\}$.
This is related to the Kodaira map $\Phi_p$ 
from \eqref{e:Kodaira}.  
Its $m$-fold product $\Phi_{p,m}$ (see \cite[Section 3.3]{DS06}) 
is the meromorphic transform from $X$ to ${\mathbb X}_{p,m}$ with graph
\[\Gamma_{p,m}=\big\{(x,s_{p1},\ldots,
s_{pm})\in X\times {\mathbb X}_{p,m}:\  s_{p1}(x)
=\ldots=s_{pm}(x)=0\big\}.
\]
Using Lemma \ref{L:Bertini} $(ii)$ and arguing as in 
\cite[Section 4.2]{CMN16}, it follows that 
$\Phi_{p,m}$ 
is a meromorphic transform of codimension $n-m$, with fibers 
\[\Phi^{-1}_{p,m}(\bfs_p)=\{x\in X:\ s_{p1}(x)=\ldots=s_{pm}(x)=0\}\,,
\,\text{ where }\,\bfs_p=(s_{p1},\ldots,s_{pm})\in{\mathbb X}_{p,m}\,.\]
Moreover, for $\bfs_p\in{\mathbb X}_{p,m}$ generic, the current 
\[\Phi_{p,m}^*(\delta_{\bfs_p})=
[\bfs_p=0]=[s_{p1}=0]\wedge\ldots\wedge[s_{pm}=0]=
\Phi_{p,1}^*(\delta_{s_{p1}})\wedge\ldots
\wedge\Phi_{p,1}^*(\delta_{s_{pm}})\]
is a well defined positive closed current of bidegree $(m,m)$ on $X$. 
Here $\delta_x$ denotes the Dirac mass at a point $x$, and $F^*(T)$ 
denotes the pull-back of a current $T$ by a meromorphic 
transform $F$ as defined in \cite[Section 3.1]{DS06}.
Following the proof of \cite[Theorem 1.2]{CM15} 
(see also \cite[Lemma 4.5]{CMN16}), we can show that 
\[\Phi_{p,m}^*(\sigma_{p,m})=\gamma_p^m\,,\,
\text{ for all $p>p_0$.}\]
We consider  the intermediate degrees of $\Phi_{p,m}$ 
of order $d_{p,m}$, resp.\ $d_{p,m}-1$ \cite[Section 3.1]{DS06}:
\begin{equation}\label{e:degree1}
\delta^1_{p,m}:=\int_X \Phi_{p,m}^*(\omega_{p,m}^{d_{p,m}})\wedge 
\omega^{n-m}\,,\,\;\delta^2_{p,m}:=
\int_X \Phi_{p,m}^*(\omega_{p,m}^{d_{p,m}-1})\wedge
\vartheta^{n-m+1}.
\end{equation}
As in the proof of \cite[Lemma 4.4]{CMN16} we obtain that 
\begin{equation}\label{e:degree2}
\delta^1_{p,m}=\int_X c_1(L_p,h_p)^m\wedge 
\vartheta^{n-m}\,,\,\;\delta^2_{p,m}=
\frac{1}{c_{p,m}}\,\int_X c_1(L_p,h_p)^{m-1}\wedge \vartheta^{n-m+1}.
\end{equation}
We will need the following estimates:
\begin{lemma}\label{L:est}

(i) For every $p\geqslant1$ and $m\in\{1,\ldots,n\}$, 
we have $\frac{1}{2em}<c_{p,m}<\frac{2e}{m}\,$.

(ii) There exist constants $M_1>1$ and $p_1>p_0$ such that, 
for every $p>p_1$, we have
\begin{align}
& M_1^{-1}A_p^n\leqslant d_p\leqslant M_1A_p^n\,,\label{e:dp}\\
M_1^{-1}A_p^m\leqslant\delta^1_{p,m}\leqslant 
M_1&A_p^m\,,\,\;M_1^{-1}A_p\leqslant
\frac{\delta^1_{p,m}}{\delta^2_{p,m}}\leqslant 
M_1A_p\,,\,\;\forall\,m\in\{1,\ldots,n\}\,.\label{e:delta}
\end{align}
\end{lemma}

\begin{proof} $(i)$ We have that \cite[p.\ 200]{Rudin76}
\[e^{^{\frac{7}{8}}}<\frac{k!}{\left(\frac{k}{e}\right)^k\sqrt{k}}\leqslant 
e\,,\,\text{ for every $k\geqslant1$.}\] 
Since $k^{^{\frac{1}{2k}}}<2$ this implies that 
$\frac{k}{e}<\big(k!\big)^{^{\frac{1}{k}}}<2k$. 
Hence by \eqref{e:c_p} and \eqref{e:d_pm},
\[\frac{1}{2em}<c_{p,m}= 
\frac{\big((d_p-1)!\big)^{1/(d_p-1)}}{\big(d_{p,m}!\big)^{1/d_{p,m}}}<
\frac{2e}{m}\,.\]
$(ii)$ We infer from \eqref{e:T} that there exists $p_1\in\N$ such that 
\begin{equation}\label{e:T1}
\frac{\omega}{2}\leqslant\frac{1}{A_p}c_1(L_p,h_p)
\leqslant 2\omega\,,\, \text{ for all $p>p_1$.}
\end{equation}
By \eqref{e:degree2} we obtain 
\[2^{-m}A_p^m\int_X\omega^m\wedge\vartheta^{n-m}
\leqslant\delta^1_{p,m}\leqslant2^mA_p^m
\int_X\omega^m\wedge\vartheta^{n-m},\]
which readily implies the first estimate from \eqref{e:delta}. 
Using this and part $(i)$, we obtain the estimate on 
$\delta^1_{p,m}/\delta^2_{p,m}$ from \eqref{e:delta}, 
by increasing the constant $M_1$. Finally, using \eqref{e:Bkf}, we get 
\[\frac{A_p^n}{M_0}\,\int_X\frac{\vartheta^n}{n!}\,
\leqslant d_p=\int_XP_p\,\frac{\vartheta^n}{n!}\leqslant 
M_0A_p^n\int_X\frac{\vartheta^n}{n!}\,,\,\text{ for all $p>p_0$.}
\]
\end{proof}

Our next result deals with the part of the proof of 
Theorem \ref{T:equidist} which uses the Dinh-Sibony 
meromorphic transform technique and equidistribution theorem 
\cite[Theorem 4.1, Lemma 4.2\,(d)]{DS06}. 
For $p>p_0$, $m\in\{1,\ldots,n\}$ and $\varepsilon>0$, let 
\begin{equation}\label{e:excset}
E_{p,m}(\varepsilon):=\bigcup_{\|\phi\|_{\cC^2}\leqslant 1}
\big\{ \bfs_p\in{\mathbb X}_{p,m}:\,\big|\big\langle[\bfs_p=0]
-\gamma_p^m,\phi\big\rangle\big|\geqslant A_p^m\varepsilon \big\},
\end{equation}
where $\phi$ is a $(n-m,n-m)$ form of class $\cC^2$ on $X$. 
We also assume that the set of $\bfs_p\in{\mathbb X}_{p,m}$ 
for which the current $[\bfs_p=0]$ is not well defined is contained in 
$E_{p,m}(\varepsilon)$. Note that, by Lemma \ref{L:Bertini}, 
the latter is a set of measure $0$ since $p>p_0$. 

\begin{prop}\label{P:MT1}
In the hypotheses of Theorem \ref{T:equidist}, there exist constants 
$\nu,\alpha,\zeta>0$ and $p_1>p_0$, such that for every $p>p_1$, 
$m\in\{1,\ldots,n\}$ and $\varepsilon>0$ we have 
\[\sigma_{p,m}(E_{p,m}(\varepsilon))\leqslant
\nu\,A_p^\zeta\,e^{-\alpha A_p\varepsilon}.\] 
\end{prop}

\begin{proof} Fix $m\in\{1,\ldots,n\}$. We apply 
	\cite[Lemma 4.2\,(d)]{DS06} 
to the sequence of meromorphic transforms 
$\Phi_{p,m}:(X,\vartheta)\dashrightarrow({\mathbb X}_{p,m},
\omega_{p,m})$ 
of codimension $n-m$ and the probability measures 
$\sigma_{p,m}=\omega_{p,m}^{d_{p,m}}$ on ${\mathbb X}_{p,m}$. Let 
\[E'_{p,m}(\varepsilon):=\bigcup_{\|\phi\|_{\cC^2}\leqslant 1}
\big\{ \bfs_p\in{\mathbb X}_{p,m}:\,\big|\big\langle[\bfs_p=0]-
\gamma_p^m,\phi\big\rangle\big|\geqslant\delta^1_{p,m}
\varepsilon \big\},\]
where $p>p_0$ and $\delta^1_{p,m}$ is the degree of 
$\Phi_{p,m}$ defined in \eqref{e:degree1}. 
By \cite[Lemma 4.2\,(d)]{DS06} 
it follows that 
\[\sigma_{p,m}(E'_{p,m}(\varepsilon))
\leqslant\Delta_p(\eta_{\varepsilon,p})\,,\,
\text{ where $\eta_{\varepsilon,p}:=
\varepsilon\,\frac{\delta^1_{p,m}}{\delta^2_{p,m}}-3R_p$.}\]
Here 
\[R_p:=R({\mathbb X}_{p,m},\omega_{p,m},\sigma_{p,m})\,,\,\;
\Delta_p(t):=   
\Delta({\mathbb X}_{p,m},\omega_{p,m},\sigma_{p,m},t), 
\text{ where $t>0$,}\] 
are quantities defined in \cite[Sections 2.1,\ 2.2]{DS06} 
and are related to certain compact classes of quasiplurisubharmonic 
functions 
on ${\mathbb X}_{p,m}$ (see also \cite[Section 4.1]{CMN16}). 
By the appendix of \cite{DS06} (see also \cite[Lemma 4.6]{CMN16}) 
we infer that 
\[R_p\leqslant\nu'm\big(1+\log d_{p,m}\big)\,,\,\;\Delta_p(t)\leqslant
\nu'\big(d_{p,m}\big)^{\zeta'}e^{-\alpha't},\;t>0,\]
where $\nu',\zeta',\alpha'>0$ are constants depending only on $m$. 

Let $M_1,p_1$ be as in Lemma \ref{L:est}. 
Then by \eqref{e:delta} we have for $p>p_1$,
\[\eta_{\varepsilon,p}\geqslant\frac{\varepsilon A_p}{M_1}-
3R_p\geqslant\frac{\varepsilon A_p}{M_1}
-3\nu'm\big(1+\log d_{p,m}\big).\]
Hence
\[\sigma_{p,m}(E'_{p,m}(\varepsilon))\leqslant
\Delta_p(\eta_{\varepsilon,p})\leqslant
\nu''\big(d_{p,m}\big)^{\zeta''}e^{-\alpha''A_p\varepsilon},\]
where $\nu'',\zeta''>0$ are constants depending only on 
$m$ and $\alpha''=\alpha'/M_1$. Using again \eqref{e:delta} 
we have $\delta^1_{p,m}\leqslant M_1A_p^m$, so 
$E_{p,m}(\varepsilon)\subset E'_{p,m}(\varepsilon/M_1)$. 
Therefore
\[\sigma_{p,m}(E_{p,m}(\varepsilon))\leqslant
\sigma_{p,m}(E'_{p,m}(\varepsilon/M_1))\leqslant
\nu''\big(d_{p,m}\big)^{\zeta''}e^{-\alpha''A_p\varepsilon/M_1}.\]
Since by \eqref{e:dp}, $d_{p,m}<md_p\leqslant mM_1A_p^n$
for $p>p_1$,  the conclusion follows.
\end{proof}

\begin{prop}\label{P:MT2}
In the hypotheses of Theorem \ref{T:equidist}, 
there exist $C>0$ and $p_1\in\N$ such that for every 
$\beta>0$, $m\in\{1,\ldots,n\}$ and $p>p_1$ there exists 
a subset $E_{p,m}^\beta\subset{\mathbb X}_{p,m}\,$ 
with the following properties:

\smallskip

(i) $\sigma_{p,m}(E_{p,m}^\beta)\leqslant CA_p^{-\beta}$;

\smallskip

(ii) if $\bfs_p\in{\mathbb X}_{p,m}\setminus E_{p,m}^\beta$ 
then, for any $(n-m,n-m)$ form $\phi$ of class $\cC^2$ on $X$,
\[\Big|\frac{1}{A_p^m}\Big\langle[\bfs_p=0]-\gamma_p^m,
\phi\Big\rangle\Big|\leqslant C\,(\beta+1)\,
\frac{\log A_p}{A_p}\,\|\phi\|_{\cC^2}\,.\]
Moreover, if $\sum_{p=1}^\infty A_p^{-\beta}<+\infty$ 
then the last estimate holds for $\sigma_{\infty,m}$-a.e.\ 
sequence $\{\bfs_p\}_{p\geqslant1}\in
{\mathbb X}_{\infty,m}$ provided that $p$ is large enough. 
\end{prop}

\begin{proof} For every $\beta>0$, $m\in\{1,\ldots,n\}$ and 
	$p>p_1$, let 
\[\varepsilon_p=\frac{(\beta+\zeta)\log A_p}{\alpha A_p}\,,
\,\;E_{p,m}^\beta:=E_{p,m}(\varepsilon_p),\]
where $p_1,\alpha,\zeta$ are as in Proposition \ref{P:MT1} 
and the set $E_{p,m}(\varepsilon)$ is defined in \eqref{e:excset}. 
By Proposition \ref{P:MT1}, we have that 
\[\sigma_{p,m}(E_{p,m}^\beta)\leqslant\nu\,A_p^\zeta\,
e^{-\alpha A_p\varepsilon_p}=\nu A_p^{-\beta}.\]
If $\bfs_p=(s_{p1},\ldots,s_{pm})\in{\mathbb X}_{p,m}\setminus 
E_{p,m}^\beta$ then, by the definition of $E_{p,m}^\beta$, the current 
$[\bfs_p=0]=[s_{p1}=0]\wedge\ldots\wedge[s_{pm}=0]$
is well defined and 
\[\Big|\frac{1}{A_p^m}\Big\langle[\bfs_p=0]-
\gamma_p^m,\phi\Big\rangle\Big|\leqslant\varepsilon_p\,
\|\phi\|_{\cC^2}\,,\]
for any $(n-m,n-m)$ form $\phi$ of class $\cC^2$. 
So assertions $(i)$ and $(ii)$ hold with the constant 
$C:=\max\big\{\nu,\frac{1}{\alpha},\frac{\zeta}{\alpha}\big\}$. 
The last assertion follows from these using the Borel-Cantelli lemma 
(see e.g.\ the proof of \cite[Theorem 4.2]{CMN16}).
\end{proof}

\begin{prop}\label{P:FSspeed}
In the hypotheses of Theorem \ref{T:equidist}, there exist 
$C>0$ and $p_1\in\N$ such that for every $m\in\{1,\ldots,n\}$, 
$p>p_1$ and every 
$(n-m,n-m)$ form $\phi$ of class $\cC^2$ on $X$, we have
\[\Big|\Big\langle\frac{\gamma_p^m}{A_p^m}-\omega^m,
\phi\Big\rangle\Big|\leqslant C\left(\frac{\log A_p}{A_p}+
A_p^{-a}\right)\|\phi\|_{\cC^2}\,.\]
\end{prop}

\begin{proof} There exists $c>0$ such that for every real 
$(n-m,n-m)$ form $\phi$ of class $\cC^2$, $m\in\{1,\ldots,n\}$, 
and every real $(1,1)$ form $\theta$ on $X$ one has 
\begin{equation}\label{e:ddc}
-c\|\phi\|_{\cC^2}\,\vartheta^{n-m+1}\leqslant dd^c\phi\leqslant 
c\|\phi\|_{\cC^2}\,\vartheta^{n-m+1},
\end{equation}
\begin{equation}\label{e:wedge}
-c\|\phi\|_{\cC^0}\|\theta\|_{\cC^0}\,\vartheta^{n-m+1}\leqslant
\phi\wedge\theta\leqslant c\|\phi\|_{\cC^0}\|\theta\|_{\cC^0}\,\vartheta^{n-m+1}.
\end{equation}
For $p>p_0$ let 
\[R_p:=\frac{\gamma_p^m}{A_p^m}-\omega^m\,,\,\;\rho_p:=
\sum_{j=0}^{m-1}\frac{\gamma_p^j}{A_p^j}\wedge 
\omega^{m-1-j}\,,\,\;\alpha_p:=\frac{c_1(L_p,h_p)}{A_p}-\omega\,.\]
By \eqref{e:T}, respectively by \eqref{e:FSB}, we have that 
\[\|\alpha_p\|_{\cC^0}\leqslant\frac{C_0}{A_p^a}\;,\qquad
\frac{\gamma_p}{A_p}-\omega=\alpha_p+\frac{1}{2A_p}\,dd^c\log P_p\,.\]
Hence if $\phi$ is a real $(n-m,n-m)$ form of class $\cC^2$ we obtain that
\begin{equation}\label{e:est0}
\langle R_p,\phi\rangle=\Big\langle\Big(\frac{\gamma_p}{A_p}-
\omega\Big)\wedge\rho_p,\phi\Big\rangle=\int_X\rho_p\wedge\alpha_p
\wedge\phi+\int_X\frac{\log P_p}{2A_p}\,\rho_p\wedge dd^c\phi\,.
\end{equation}
Using \eqref{e:wedge} we infer that
\[-\frac{c\,C_0}{A_p^a}\,\|\phi\|_{\cC^0}\,
\vartheta^{n-m+1}\leqslant\alpha_p\wedge
\phi\leqslant\frac{cC_0}{A_p^a}\,\|\phi\|_{\cC^0}\,\vartheta^{n-m+1},\]
hence 
\begin{equation}\label{e:est1}
\Big|\int_X\rho_p\wedge\alpha_p\wedge\phi\Big|
\leqslant\frac{cC_0}{A_p^a}\,\|\phi\|_{\cC^0}\int_X\rho_p
\wedge\vartheta^{n-m+1}.
\end{equation}
By \eqref{e:ddc}, the total variation of the signed measure 
$\rho_p\wedge dd^c\phi$ verifies 
\[|\rho_p\wedge dd^c\phi|\leqslant c\|\phi\|_{\cC^2}\,
\rho_p\wedge\vartheta^{n-m+1}.\] 
Therefore 
\[\Big|\int_X\frac{\log P_p}{2A_p}\,\rho_p\wedge dd^c
\phi\Big|\leqslant c\|\phi\|_{\cC^2}\int_X\frac{|\log P_p|}{2A_p}\,
\rho_p\wedge\vartheta^{n-m+1}.\]

We choose $p_1>p_0$ such that \eqref{e:T1} holds for $p>p_1$ 
and $A_p>M_0$ for $p>p_1$. By \eqref{e:Bkf} it follows that 
$A_p^{n-1}\leqslant P_p\leqslant A_p^{n+1}$, so 
$|\log P_p|\leqslant(n+1)\log A_p$, hold on $X$ for $p>p_1$. 
We infer that 
\begin{equation}\label{e:est2}
\Big|\int_X\frac{\log P_p}{2A_p}\,\rho_p\wedge dd^c\phi\Big|
\leqslant\frac{nc\|\phi\|_{\cC^2}\log A_p}{A_p}\int_X\rho_p\wedge
\vartheta^{n-m+1}\,\text{ for $p>p_1$.}
\end{equation}
Using \eqref{e:T1} and \eqref{e:FSB} we have, for $p>p_1$ and 
$0\leqslant j\leqslant m-1$, that
\[\int_X\frac{\gamma_p^j}{A_p^j}\wedge \omega^{m-1-j}
\wedge\vartheta^{n-m+1}=\int_X\frac{c_1(L_p,h_p)^j}{A_p^j}
\wedge \omega^{m-1-j}\wedge\vartheta^{n-m+1}
\leqslant2^j\int_X\omega^{m-1}\wedge\vartheta^{n-m+1}.\]
Hence
\begin{equation}\label{e:est3}
\int_X\rho_p\wedge\vartheta^{n-m+1}
=\sum_{j=0}^{m-1}\int_X\frac{\gamma_p^j}{A_p^j}\wedge
\omega^{m-1-j}\wedge\vartheta^{n-m+1}<2^m\int_X\omega^{m-1}
\wedge\vartheta^{n-m+1}.
\end{equation}
By \eqref{e:est0}, \eqref{e:est1}, \eqref{e:est2} and \eqref{e:est3} 
we conclude that if $p>p_1$ then 
\[
|\langle R_p,\phi\rangle|\leqslant 2^m\Big(\frac{c\,C_0}{A_p^a}\,
\|\phi\|_{\cC^0}+\frac{nc\log A_p}{A_p}\,\|\phi\|_{\cC^2}\Big)
\int_X\omega^{m-1}\wedge\vartheta^{n-m+1},
\]
for every $m\in\{1,\ldots,n\}$ and every real  $(n-m,n-m)$ form 
$\phi$ of class $\cC^2$. This implies the proposition.
\end{proof}

\smallskip

\begin{proof}[Proof of Theorem \ref{T:equidist}] 
	Theorem \ref{T:equidist} follows at once from 
Propositions \ref{P:MT2} and \ref{P:FSspeed}.
\end{proof}

\end{document}